\documentclass[12pt]{article}%
\usepackage{amsmath}
\usepackage{amsfonts}
\usepackage{amssymb}
\usepackage{graphicx}
\usepackage{color}%
\providecommand{\U}[1]{\protect\rule{.1in}{.1in}}
\newtheorem{theorem}{Theorem}[section]

\newtheorem{corollary}[theorem]{Corollary}

\newtheorem{example}[theorem]{Example}

\newtheorem{lemma}[theorem]{Lemma}

\newtheorem{remark}[theorem]{Remark}

\newtheorem{assumption}[theorem]{Assumption}
\newcommand{\BIGOP}[1]{\mathop{\mathchoice{\raise-0.22em\hbox{\huge
$#1$}} {\raise-0.05em\hbox{\Large $#1$}}{\hbox{\large $#1$}}{#1}}}
\newenvironment{proof}[1][Proof]{\textbf{#1.} }{\ \rule{0.5em}{0.5em}}

\makeatletter\@addtoreset{equation}{section}\makeatother
\evensidemargin=\oddsidemargin
\newdimen\dummy
\dummy=\oddsidemargin
\begin{document}

\title{On some nonlocal, nonlinear diffusion problems}
\author{M.M. Chipot\thanks{(m.m.chipot@math.uzh.ch), Institut f\"{u}r Mathematik,
Universit\"{a}t Z\"{u}rich, Winterthurerstr 190, CH-8057 Z\"{u}rich,
Switzerland}
\and A. Luthra\thanks{(aryan.luthra@uzh.ch), Institut f\"{u}r Mathematik,
Universit\"{a}t Z\"{u}rich, Winterthurerstr 190, CH-8057 Z\"{u}rich,
Switzerland }
\and S.A. Sauter\thanks{(stas@math.uzh.ch), Institut f\"{u}r Mathematik,
Universit\"{a}t Z\"{u}rich, Winterthurerstr 190, CH-8057 Z\"{u}rich,
Switzerland}}
\maketitle

\begin{abstract}
This note is devoted to some nonlocal, nonlinear elliptic problems with an
emphasis on the computation of the solution of such problems, reducing it in
particular to a fixed point argument in $\mathbb{R}$. Errors estimates and
numerical experiments are provided.

\end{abstract}



\section{Introduction}

It is well known that the density of a population $u=u(t,x)$ diffusing in a
domain $\Omega$ satisfies the equation
\begin{equation}%
\begin{cases}
\partial_{t}u-a\Delta u+\lambda u=F\text{ in }\Omega,\cr u(0,x)=u_{0}(x)\text{
in }\Omega\cr
\end{cases}
\label{1.1}%
\end{equation}
together with a boundary condition, for instance
\begin{equation}
u(t,x)=0,x\in\partial\Omega\label{1.2}%
\end{equation}
if the boundary of the domain is inhospitable. In the above, $\partial\Omega$
denotes the boundary of $\Omega$, $u_{0}$ the initial density of population at
time $t=0$, $a$ the diffusion coefficient, $\lambda$ the death rate of the
species at stake and $F$ its supply. To consider $a$ and $\lambda$ constant is
a first approximation and it is more realistic to have these coefficients
depending on $x$. Furthermore the amount of individuals might have some
influence on the decision to move and thus a more realistic model would
request to have
\begin{equation}
a=a(x,u),~~\lambda=\lambda(x,u). \label{1.3}%
\end{equation}
In this definition the dependence in $u$ is not very explicit. For instance
one could have
\begin{equation}
a=a(x,u(x)),~~\lambda=\lambda(x,u(x)), \label{1.4}%
\end{equation}
and these coefficients would depend then from the population $u(x)$ at the
location $x$. Again, this might lack of realism and, for instance, the whole
population might influence the mobility of the crowd in such a way that,
dropping the measure of integration,
\begin{equation}
a=a(x,\int_{\Omega}u),~~\lambda=\lambda(x,\int_{\Omega}u). \label{1.5}%
\end{equation}
This corresponds to a so called nonlocal problem whereas \eqref{1.4} is a
local assumption. For the study of such problems, where $u$ could be some
other physical quantity, the reader is referred to \cite{CNew}, \cite{DL},
\cite{Evans}, \cite{QS}.

When the population has reached stability -i.e. when $\partial_{t}u=0$- the
equations reduce to
\begin{equation}%
\begin{cases}
-a(x,u)\Delta u+\lambda(x,u)u=F\text{ in }\Omega,\cr u(x)=0,x\in\partial
\Omega.\cr
\end{cases}
\label{1.6}%
\end{equation}
These are problems that we would like to investigate from the point of view of
existence and uniqueness of a solution oriented toward its computation. More
abstract results can be found in \cite{CR}, where multivalued $\ell$ are
considered, or in \cite{C1}.

The paper is organised as follows. In the next section we address the issue of
existence of a solution for coefficients of the type \eqref{1.5} showing in
particular its equivalence to a fixed point for an equation in $\mathbb{R}$.
In section 3 we treat the case when the dependence of $a$ and $\lambda$ is
less explicit in the spirit of \eqref{1.3}. The section 4 addresses the issue
of numerical computations, section 5 develops a corresponding error analysis
and finally the last section is devoted to numerical experiments. In the
appendix, we analyse a modified fixed point iteration with improved
convergence properties under certain assumptions.

\section{A first class of elliptic problems}

Let $\Omega\subset\mathbb{R}^{n}$ be a bounded Lipschitz domain with boundary
$\partial\Omega$. The Euclidean scalar product in $\mathbb{R}^{n}$ is denoted
by $\left\langle \cdot,\cdot\right\rangle $, the Euclidean norm by $\left\vert
\cdot\right\vert $, and the $L^{2}\left(  \Omega\right)  $ scalar product by
$\left(  \cdot,\cdot\right)  _{\Omega}$. We use the standard notation for the
Sobolev space $H^{1}\left(  \Omega\right)  $ and the subset $H_{0}^{1}\left(
\Omega\right)  \subset H^{1}\left(  \Omega\right)  $ consisting of function
with zero trace on $\partial\Omega$. The dual space of $H_{0}^{1}\left(
\Omega\right)  $ is denoted by $H^{-1}\left(  \Omega\right)  $. The space
$H_{0}^{1}\left(  \Omega\right)  $ is equipped with the norm $\left\Vert
v\right\Vert =\left(  \nabla v,\nabla v\right)  _{\Omega}^{1/2}$ and for
$H^{-1}\left(  \Omega\right)  $ the strong dual norm is denoted by $\left\Vert
F\right\Vert _{\ast}$.

Let $\mathbb{A}:\Omega\times\mathbb{R}\rightarrow\mathbb{R}^{n\times n}$ be
the matrix-valued diffusion coefficient which satisfies for some $\alpha
,\beta>0$ and every $%
\mbox{\boldmath$ \xi$}%
\in\mathbb{R}^{n}$%
\begin{equation}
\alpha\left\vert
\mbox{\boldmath$ \xi$}%
\right\vert ^{2}\leq\left\langle \mathbb{A}\left(  \mathbf{x},r\right)
\mbox{\boldmath$ \xi$}%
,%
\mbox{\boldmath$ \xi$}%
\right\rangle \leq\beta\left\vert
\mbox{\boldmath$ \xi$}%
\right\vert ^{2}\quad\text{a.e. }\mathbf{x}\in\Omega\text{, }\forall
r\in\mathbb{R}. \label{mc1}%
\end{equation}
We suppose that the functions $\mathbb{A}\left(  x,r\right)  $ and
$\lambda\left(  x,r\right)  $ are Carath\'{e}odory functions that is
\begin{equation}
\begin{aligned} \text{ for every } r \in\mathbb{R},~~ \mathbb{A}\left( \cdot,r\right), \lambda\left( \cdot,r\right) \text{ are measurable},\cr \text{ for almost every } \mathbf{x} \in \Omega,~~ \mathbb{A}\left(x, \cdot\right), \lambda \left(\mathbf{x}, \cdot \right) \text{ are continuous.}\cr \end{aligned} \label{mc2.5}%
\end{equation}
In addition $\lambda$ satisfies%

\begin{equation}
\label{mc2.6}0\leq\lambda\left(  \mathbf{x},r\right)  \leq\beta\quad\text{a.e.
}\mathbf{x}\in\Omega\text{, }\forall r\in\mathbb{R}.
\end{equation}

Let us denote by $\ell$ a mapping from $H_{0}^{1}\left(  \Omega\right)  $ into
$\mathbb{R}$ and by $S$ a closed subspace of $H_{0}^{1}\left(  \Omega\right)
$ .

Our goal is to investigate the problem: for given $F\in H^{-1}\left(
\Omega\right)  $ find $u=u_{S}$ such that
\begin{equation}
\label{mc3}%
\begin{cases}
u\in S,\cr \int_{\Omega} \left\langle \mathbb{A}\left(  \cdot,\ell\left(
u\right)  \right)  \nabla u,\nabla v\right\rangle +\lambda\left(  \cdot
,\ell\left(  u\right)  \right)  uv =F\left(  v\right)  \quad\forall v\in S.
\cr
\end{cases}
\end{equation}


Related to this problem is the following one parameter-dependent Poisson
equation: for $\mu\in\mathbb{R}$, find $u=u_{S,\mu}$ such that
\begin{equation}
\label{mc3.6}%
\begin{cases}
u_{S,\mu}\in S,\cr \int_{\Omega} \left\langle \mathbb{A}\left(  \cdot
,\mu\right)  \nabla u_{S,\mu},\nabla v\right\rangle +\lambda\left(  \cdot
,\mu\right)  u_{S,\mu}v =F\left(  v\right)  \quad\forall v\in S. \cr
\end{cases}
\end{equation}

\begin{remark}
The assumptions on the coefficients in (\ref{mc3.6}) imply via the Lax-Milgram
lemma the well posedness of it.

\end{remark}

An equivalence relation between these two problems is stated in the following theorem.

\begin{theorem}
\label{mcThm1}There is a one-to-one mapping ($u\longmapsto\ell\left(
u\right)  $) from the set of solutions to (\ref{mc3}) onto the set of
solutions to the fixed point equation in $\mathbb{R}$:%
\begin{equation}
\mu=\ell\left(  u_{S,\mu}\right)  . \label{mc4}%
\end{equation}

\end{theorem}

%

\proof
Suppose that $u$ is a solution to (\ref{mc3}). Then clearly,%
\[
u=u_{S,\ell\left(  u\right)  }\quad\text{and\quad}\ell\left(  u\right)
=\ell\left(  u_{S,\ell\left(  u\right)  }\right)  ,
\]
i.e., $\mu=\ell\left(  u\right)  $ is the solution to (\ref{mc4}). Conversely
if $\mu$ is a solution to (\ref{mc4}) then $u_{\mu}$ is a solution to
(\ref{mc3}).%
\endproof
\vskip .3 cm It is remarkable that \eqref{mc4} allows to solve \eqref{mc3}
with almost no functional analysis. To see this let us first establish the
following lemma.

\begin{lemma}
Under the above assumptions the mapping
\begin{equation}
\label{mc4.1}\mu\longmapsto u_{S,\mu}%
\end{equation}
is continuous from $\mathbb{R}$ into $H^{1}_{0}(\Omega)$.
\end{lemma}

\begin{proof}
Consider a sequence $\mu_{n}$ such that $\mu_{n} \to\mu$ when $n\to+\infty.$
One has by \eqref{mc3.6}%

\begin{equation}
\label{mc4.2}\int_{\Omega} \left\langle \mathbb{A}\left(  \cdot,\mu_{n}
\right)  \nabla u_{S,\mu_{n}},\nabla v\right\rangle +\lambda\left(  \cdot
,\mu_{n}\right)  u_{S,\mu_{n}}v =F\left(  v\right)  \quad\forall v\in S.
\end{equation}
Taking $v = u_{S,\mu_{n}}$ we deduce by \eqref{mc1}, \eqref{mc2.6}
\begin{equation}
\label{mc4.3}\alpha\int_{\Omega} \left\vert \nabla u_{S,\mu_{n}} \right\vert
^{2} \leq F\left(  u_{S,\mu_{n}}\right)  \leq\left\Vert F\right\Vert _{\ast}
\Big( \int_{\Omega} \left\vert \nabla u_{S,\mu_{n}} \right\vert ^{2}
\Big)^{\frac{1}{2}}%
\end{equation}
and thus
\begin{equation}
\label{mc4.4}\int_{\Omega} \left\vert \nabla u_{S,\mu_{n}} \right\vert ^{2}
\leq\frac{ \left\Vert F\right\Vert _{\ast}^{2} }{\alpha^{2}}.
\end{equation}
Since $u_{S,\mu_{n}}$ is bounded in $H^{1}_{0}(\Omega)$ - up to a subsequence
- one has for some $u \in H^{1}_{0}(\Omega)$
\[
\nabla u_{S,\mu_{n}} \rightharpoonup\nabla u \text{ in } ({L}^{2}(\Omega
))^{n}~~,~~ u_{S,\mu_{n}} \to u \text{ in } ({L}^{2}(\Omega) .
\]
For any $v\in S$ it holds
\[
\mathbb{A}\left(  \cdot,\mu_{n} \right)  \nabla v \to\mathbb{A}\left(
\cdot,\mu\right)  \nabla v \text{ in } ({L}^{2}(\Omega))^{n}~~,~~
\lambda\left(  \cdot,\mu_{n}\right)  v \to\lambda\left(  \cdot,\mu\right)  v
\text{ in } {L}^{2}(\Omega) .
\]
Passing to the limit in \eqref{mc4.2} one sees that $u$ satisfies
\begin{equation}
\label{mc4.5}\int_{\Omega} \left\langle \mathbb{A}\left(  \cdot,\mu\right)
\nabla u,\nabla v\right\rangle +\lambda\left(  \cdot,\mu\right)  uv =F\left(
v\right)  \quad\forall v\in S,
\end{equation}
i.e., $u=u_{S,\mu}$. Since the limit is uniquely determined one has
\begin{equation}
\label{mc}u_{S,\mu_{n}} \rightharpoonup u_{S,\mu} \text{ in } H^{1}_{0}%
(\Omega)~~,~~ u_{S,\mu_{n}} \to u_{S,\mu} \text{ in } {L}^{2}(\Omega) .
\end{equation}
To obtain the strong convergence in $H^{1}_{0}(\Omega)$ one notices that
\begin{equation}
\label{mcc}\begin{aligned} \alpha& \int_{\Omega} \left\vert \nabla ( u_{S,\mu_n} - u_{S,\mu}) \right\vert^2 \cr & \leq \int_{\Omega} \left\langle \mathbb{A}\left( \cdot,\mu_n \right) \nabla ( u_{S,\mu_n} - u_{S,\mu}),\nabla ( u_{S,\mu_n} - u_{S,\mu})\right\rangle \cr &~~~~~~~~~~~~~~~~~~~~~~~~~~~~~~~+\lambda\left( \cdot ,\mu_n\right) ( u_{S,\mu_n} - u_{S,\mu})( u_{S,\mu_n} - u_{S,\mu})\cr &= F( u_{S,\mu_n} - u_{S,\mu}) \cr &~~~~~~~~~~~~ - \int_{\Omega} \left\langle \mathbb{A}\left( \cdot,\mu_n \right) \nabla u_{S,\mu},\nabla ( u_{S,\mu_n} - u_{S,\mu})\right\rangle + \lambda\left( \cdot ,\mu_n\right) u_{S,\mu}( u_{S,\mu_n} - u_{S,\mu})\cr \end{aligned}
\end{equation}
and this last quantity goes to 0 when $n\to+\infty$.
\end{proof}

\vskip .5 cm Then one has the following very simple existence result.

\begin{theorem}
Suppose that $\ell: S \to\mathbb{R}$ is continuous and bounded on bounded sets
of $S$, then \eqref{mc3} admits at least one solution.
\end{theorem}

\begin{proof}
From \eqref{mc3.6} one deduces as in \eqref{mc4.3}, \eqref{mc4.4} that
$u_{S,\mu}$ is bounded independently of $\mu$ in $H^{1}_{0}(\Omega)$. Then due
to the preceding lemma and our assumption on $\ell$ the function
$\mu\longmapsto\ell(u_{S,\mu})$ is continuous and bounded uniformly in $\mu$.
Thus the straight line $y = \mu$ is cutting at at least one point the graph of
the curve $y = \ell(u_{S,\mu})$.
\end{proof}

\vskip .5 cm The following situation is especially interesting.

\begin{corollary}
\label{CorSimplCoeff}Suppose that%
\begin{equation}
\lambda=0\quad\text{and\quad}\mathbb{A}\left(  \mathbf{x},r\right)  =a\left(
r\right)  \operatorname*{Id}\quad\text{a.e. }\mathbf{x}\in\Omega,\ \forall
r\in\mathbb{R}, \label{mc5}%
\end{equation}
where $\operatorname*{Id}$ denotes the $n\times n$ identity matrix and that
$\ell\left(  \cdot\right)  $ is homogeneous of degree $p>0$, i.e.,%
\begin{equation}
\ell\left(  \alpha u\right)  =\alpha^{p}\ell\left(  u\right)  \qquad
\forall\alpha\geq0,\ \ \forall u\in S. \label{mc6}%
\end{equation}
Then, there is a one-to-one mapping ($u\longmapsto\ell\left(  u\right)  $)
between the solutions to (\ref{mc3}) and the solution of the equation%
\begin{equation}
\mu=\frac{\ell\left(  \psi_{S}\right)  }{a\left(  \mu\right)  ^{p}},
\label{FPsimplified}%
\end{equation}
where $\psi_{S}\in S\ $ is the weak solution of the Poisson problem:%
\begin{equation}%
\begin{cases}
\psi_{S}\in S,\cr\int_{\Omega}\left\langle \nabla\psi_{S},\nabla
v\right\rangle =F\left(  v\right)  \quad\forall v\in S.\label{linPoisson}%
\end{cases}
\end{equation}

\end{corollary}

%

\proof
For $\mu\in\mathbb{R}$, $u_{S,\mu}$ is the solution to
\[%
\begin{cases}
u_{S,\mu} \in S, \cr \int_{\Omega}a(\mu)\left\langle \nabla u_{S,\mu},\nabla
v\right\rangle =F\left(  v\right)  \quad\forall v\in S,
\end{cases}
\]

i.e. one has $a\left(  \mu\right)  u_{S,\mu}=\psi_{S}$. We apply $\ell\left(
\cdot\right)  $ to both sides and employ (\ref{mc6}) to obtain%
\[
a\left(  \mu\right)  ^{p}\ell\left(  u_{S\mu}\right)  =\ell\left(  \psi
_{S}\right)  .
\]
The result then follows from (\ref{mc4}).%
\endproof

\begin{example}
Some suitable functionals $\ell\left(  \cdot\right)  $ are given by%
\begin{align*}
\ell\left(  u\right)  =  &  \int_{\Omega}u,\quad\int_{\Omega}\left\vert
u\right\vert ^{p}\quad\text{for }1\leq p<p^{\ast}\text{ and }p^{\ast}\text{
such that }H_{0}^{1}\left(  \Omega\right)  \hookrightarrow L^{p^{\ast}}\left(
\omega\right)  ,\\
&  \int_{\Omega}\left\vert \nabla u\right\vert ^{2},\quad\int_{\Omega^{\prime
}}u,\quad\int_{\Omega^{\prime}}\left\vert \nabla u\right\vert ^{2},\quad
\ldots\quad\text{for }\Omega^{\prime}\subset\Omega.
\end{align*}

\end{example}

\section{A more general class of problems}

Of course, existence and uniqueness results can be achieved just by our
derivation of (\ref{mc4}). However, we will embed the above problems in a more
general class of problems which includes in particular the Galerkin
approximation of such problems.

Let $\mathbb{L}^{\infty}\left(  \Omega\right)  $ be the set of functions from
$\Omega$ to $\mathbb{R}^{n\times n}$ whose coefficients belong to $L^{\infty
}\left(  \Omega\right)  $. We assume that $\mathbb{A}:H_{0}^{1}\left(
\Omega\right)  \rightarrow\mathbb{L}^{\infty}\left(  \Omega\right)  $ and
$\lambda:H_{0}^{1}\left(  \Omega\right)  \rightarrow L^{\infty}\left(
\Omega\right)  $ are continuous on finite dimensional subspaces of $H_{0}%
^{1}\left(  \Omega\right)  $ and that there exist constants $\alpha,\beta>0$
such that\footnote{Note that the diffusion and the reaction coefficients in
(\ref{mc3}) induce a diffusion and reaction coefficient in the new notation
(\ref{mc1prime}) via $\mathbb{A}\left(  u\right)  \left(  \mathbf{x}\right)
\leftarrow\mathbb{A}\left(  \mathbf{x},\ell\left(  u\right)  \right)  $ and
$\lambda\left(  u\right)  \left(  \mathbf{x}\right)  \leftarrow\lambda\left(
\mathbf{x},\ell\left(  u\right)  \right)  $.}%
\begin{equation}
\left.
\begin{array}
[c]{c}%
\alpha\left\vert
\mbox{\boldmath$ \xi$}%
\right\vert ^{2}\leq\left\langle \mathbb{A}\left(  u\right)  \left(
\mathbf{x}\right)
\mbox{\boldmath$ \xi$}%
,%
\mbox{\boldmath$ \xi$}%
\right\rangle \leq\beta\left\vert
\mbox{\boldmath$ \xi$}%
\right\vert ^{2}\\
\beta\geq\lambda\left(  u\right)  \left(  \mathbf{x}\right)  \geq0
\end{array}
\right\}  \quad\left\{
\begin{array}
[c]{l}%
\text{a.e. }\mathbf{x}\in\Omega,\\
\forall%
\mbox{\boldmath$ \xi$}%
\in\mathbb{R}^{n},\\
\forall u\in H_{0}^{1}\left(  \Omega\right)  .
\end{array}
\right.  \label{mc1prime}%
\end{equation}
Denote by $S$ a finite dimensional subspace of $H_{0}^{1}\left(
\Omega\right)  $. We consider the Galerkin discretization:%
\begin{equation}
\text{find }u_{S}\in S\text{ such that}\quad\int_{\Omega} \left\langle
\mathbb{A}\left(  u_{S}\right)  \nabla u_{S},\nabla v\right\rangle
+\lambda\left(  u_{S}\right)  u_{S}v =F\left(  v\right)  \quad\forall v\in S.
\label{mc7}%
\end{equation}
The existence of a solution to (\ref{mc7}) is stated in the following theorem.

\begin{theorem}
\label{mcThm2}Let (\ref{mc1prime}) be satisfied. For any $F\in H^{-1}\left(
\Omega\right)  $, problem (\ref{mc7}) has a solution.
\end{theorem}

%

\proof
Let $w\in S$. By our assumptions there exists a unique $u_{S}=T\left(
w\right)  $ which is the solution to%
\begin{equation}
\text{find }u_{S}\in S\text{ such that}\quad\int_{\Omega} \left\langle
\mathbb{A}\left(  w\right)  \nabla u_{S},\nabla v\right\rangle +\lambda\left(
w\right)  u_{S}v =F\left(  v\right)  \quad\forall v\in S. \label{mc8}%
\end{equation}
We claim that $T$ maps%
\[
B:=\left\{  v\in S\mid\left\Vert v\right\Vert \leq\left\Vert F\right\Vert
_{\ast}/\alpha\right\}
\]
into itself. Indeed, if $w\in B$ one has by taking $v=u_{S}$ in (\ref{mc8})
the estimate%
\begin{equation}
\alpha\left\Vert u_{S}\right\Vert ^{2}\leq F\left(  u_{S}\right)
\leq\left\Vert F\right\Vert _{\ast}\left\Vert u_{S}\right\Vert , \label{mc9}%
\end{equation}
i.e., $u_{S}\in B$.

Next, we prove that $T$ is continuous. Consider a sequence $\left(
w_{n}\right)  _{n}$ in $S$ which converges to $w\in S$ with respect to the
norm $\left\Vert \cdot\right\Vert $. Let $u_{S,n}:=T\left(  w_{n}\right)  $.
Since $\left(  u_{S,n}\right)  _{n}$ is uniformly bounded, one has for a
subsequence:%
\[
u_{S,n}\rightharpoonup u_{S}\quad\text{in }H_{0}^{1}\left(  \Omega\right)
,\quad u_{S,n}\rightarrow u_{S}\quad\text{in }L^{2}\left(  \Omega\right)
\]
for some $u_{S}\in S$, i.e., $\nabla u_{S,n}\rightharpoonup\nabla u_{S}$ in
$\left(  L^{2}\left(  \Omega\right)  \right)  ^{n}$. The function $u_{S,n}$
satisfies%
\begin{equation}
\int_{\Omega} \left\langle \mathbb{A}\left(  w_{n}\right)  \nabla
u_{S,n},\nabla v\right\rangle +\lambda\left(  w_{n}\right)  u_{S,n}v =F\left(
v\right)  \quad\forall v\in S. \label{mc10}%
\end{equation}
Due to our continuity assumptions on $\mathbb{A}$ and $\lambda$ one has%
\[
\mathbb{A}\left(  w_{n}\right)  \nabla v\rightarrow\mathbb{A}\left(  w\right)
\nabla v\quad\text{in }\left(  L^{2}\left(  \Omega\right)  \right)  ^{n}%
,\quad\lambda\left(  w_{n}\right)  v\rightarrow\lambda\left(  w\right)
v\quad\text{in }L^{2}\left(  \Omega\right)  .
\]
By passing to the limit in (\ref{mc10}) we obtain%
\[
\int_{\Omega} \left\langle \mathbb{A}\left(  w\right)  \nabla u_{S},\nabla
v\right\rangle +\lambda\left(  w\right)  u_{S}v =F\left(  v\right)
\quad\forall v\in S
\]
and, in turn, $u_{S}=T\left(  w\right)  $. Since the limit is uniquely
determined the whole sequence $u_{S,n}$ converges toward $u_{S}=T\left(
w\right)  $. In this way it is proved that $T$ is a continuous mapping from
$B$ into itself and it has a fixed point as a consequence of the Brouwer fixed
point theorem.%
\endproof

\vskip .3 cm For Lipschitz continuous diffusion coefficients $\mathbb{A}$ and
$\lambda=0$, the following uniqueness result holds.

\begin{theorem}
\label{mcThm3}Suppose that $\lambda=0$ and $\mathbb{A}$ is such
that\footnote{For matrix $\mathbb{A}\in\mathbb{R}^{n\times n}$, the norm
$\left\Vert \cdot\right\Vert _{2}$ is given by $\left\Vert \mathbb{A}%
\right\Vert _{2}:=\sup\left\{  \left\vert \mathbb{A}\mathbf{v}\right\vert
/\left\vert \mathbf{v}\right\vert :\mathbf{v}\in\mathbb{R}^{n}\backslash
\left\{  0\right\}  \right\}  $.}%
\begin{equation}
\left\Vert \mathbb{A}\left(  u\right)  \left(  \mathbf{x}\right)
-\mathbb{A}\left(  \tilde{u}\right)  \left(  \mathbf{x}\right)  \right\Vert
_{2}\leq\gamma\left\Vert u-\tilde{u}\right\Vert \quad\text{a.e. }\mathbf{x}%
\in\Omega\quad\forall u,\tilde{u}\in H_{0}^{1}\left(  \Omega\right)  .
\label{mc11}%
\end{equation}
Let the following smallness assumption be satisfied%
\begin{equation}
\gamma\left\Vert F\right\Vert _{\ast}<\alpha^{2}. \label{mc12}%
\end{equation}
Then, the solution to (\ref{mc7}) is unique.
\end{theorem}

%

\proof
Let $u_{S}$, $\tilde{u}_{S}$ be two solutions to (\ref{mc7}). One has%
\[
\int_{\Omega}\left\langle \mathbb{A}\left(  u_{S}\right)  \nabla u_{S},\nabla
v\right\rangle =\int_{\Omega}\left\langle \mathbb{A}\left(  \tilde{u}%
_{S}\right)  \nabla\tilde{u}_{S},\nabla v\right\rangle \quad\forall v\in S
\]
which can be rewritten as%
\begin{equation}
\int_{\Omega}\left\langle \mathbb{A}\left(  u_{S}\right)  \nabla\left(
u_{S}-\tilde{u}_{S}\right)  ,\nabla v\right\rangle =\int_{\Omega}\left\langle
\left(  \mathbb{A}\left(  \tilde{u}_{S}\right)  -\mathbb{A}\left(
u_{S}\right)  \right)  \nabla\tilde{u}_{S},\nabla v\right\rangle \quad\forall
v\in S. \label{mc12.5}%
\end{equation}
For the choice $v=u_{S}-\tilde{u}_{S}$ we get%
\begin{align*}
\alpha\int_{\Omega}\left\vert \nabla\left(  u-\tilde{u}_{S}\right)
\right\vert ^{2}  &  \leq\int_{\Omega}\left\Vert \mathbb{A}\left(  \tilde
{u}_{S}\right)  -\mathbb{A}\left(  u_{S}\right)  \right\Vert _{2}\left\vert
\nabla\tilde{u}_{S}\right\vert \left\vert \nabla\left(  u_{S}-\tilde{u}%
_{S}\right)  \right\vert \\
&  \leq\gamma\left\Vert u_{S}-\tilde{u}_{S}\right\Vert \int_{\Omega}\left\vert
\nabla\tilde{u}_{S}\right\vert \left\vert \nabla\left(  u_{S}-\tilde{u}%
_{S}\right)  \right\vert .
\end{align*}
By a Cauchy-Schwarz inequality and (\ref{mc9}), i.e., $\alpha\left\Vert
u_{S}\right\Vert \leq\left\Vert F\right\Vert _{\ast}$ it follows%
\[
\alpha\left\Vert u-u_{S}\right\Vert ^{2}\leq\gamma\left\Vert u-u_{S}%
\right\Vert ^{2}\left\Vert u_{S}\right\Vert \leq\frac{\gamma}{\alpha
}\left\Vert u-u_{S}\right\Vert ^{2}\left\Vert F\right\Vert _{\ast}%
\]
and necessarily $u=u_{S}$ if $\gamma\left\Vert F\right\Vert _{\ast}/\alpha
^{2}<1$.
\endproof

\begin{remark}
Note that the theorem remains valid with a local Lipschitz continuity
assumption \eqref{mc11} since the solutions are a priori bounded. Also to pass
to the limit in (\ref{mc10}) one only needs the following continuity%
\begin{equation}
w_{S}\rightarrow w\quad\text{in }L^{2}\left(  \Omega\right)  \implies
\mathbb{A}\left(  w_{S}\right)  ,\ \lambda\left(  w_{S}\right)  \rightarrow
\mathbb{A}\left(  w\right)  ,\ \lambda\left(  w\right)  \quad\text{a.e. in
}\Omega. \label{mc13}%
\end{equation}
Note also that with an easy modification the theorem is valid for $\lambda$
such that $u\rightarrow\lambda(u)$ is monotone. Indeed in this case
\eqref{mc12.5} is replaced by
\begin{equation}
\begin{aligned} &\int_{\Omega}\left\langle \mathbb{A}\left( u_{S}\right) \nabla\left( u_{S}-\tilde{u}_{S}\right) ,\nabla v\right\rangle \cr &~~~~~~~~~~~~~~~~~~~~~=\int_{\Omega}\left\langle \left( \mathbb{A}\left( \tilde{u}_{S}\right) -\mathbb{A}\left( u_{S}\right) \right) \nabla\tilde{u}_{S},\nabla v\right\rangle -\{ \lambda ({u}_{S}) - \lambda (\tilde{u}_{S})\}v \quad\forall v\in S.\cr \end{aligned}
\end{equation}
Taking $v=u_{S}-\tilde{u}_{S}$ would allow to conclude as above since the
monotonicity implies
\[
\int_{\Omega}( \lambda({u}_{S}) - \lambda(\tilde{u}_{S}))(u_{S}-\tilde{u}_{S})
\geq0.
\]

\end{remark}

Using Theorem \ref{mcThm2} we can deduce a global result of existence for the
problem: given $F\in H^{-1}\left(  \Omega\right)  $:%
\begin{equation}
\text{find }u\in H_{0}^{1}\left(  \Omega\right)  \text{ such that}\quad
\int_{\Omega} \left\langle \mathbb{A}\left(  u\right)  \nabla u,\nabla
v\right\rangle +\lambda\left(  u\right)  uv =F\left(  v\right)  \quad\forall
v\in H_{0}^{1}\left(  \Omega\right)  . \label{mc14}%
\end{equation}

\begin{theorem}
\label{mcThm4}Suppose $\mathbb{A}$, $\lambda$ are functions from $H_{0}%
^{1}\left(  \Omega\right)  $ into $\mathbb{L}^{\infty}\left(  \Omega\right)
$, $L^{\infty}\left(  \Omega\right)  $ respectively satisfying the assumptions
(\ref{mc1prime}) and (\ref{mc13}). Then for $F\in H^{-1}\left(  \Omega\right)
$ there exists a solution $u$ to (\ref{mc14}).
\end{theorem}

%

\proof
Denote by $\left(  S_{n}\right)  _{n}$ a sequence of finite dimensional spaces
of $H_{0}^{1}\left(  \Omega\right)  $ such that $%
{\displaystyle\bigcup_{n}}
S_{n}$ is dense in $H_{0}^{1}\left(  \Omega\right)  $. Theorem \ref{mcThm2}
implies that there exists a solution to
\begin{equation}
\text{find }u_{n}\in S_{n}\text{ such that}\quad\int_{\Omega} \left\langle
\mathbb{A}\left(  u_{n}\right)  \nabla u_{n},\nabla v\right\rangle
+\lambda\left(  u_{n}\right)  u_{n}v=F\left(  v\right)  \quad\forall v\in
S_{n}. \label{mc15}%
\end{equation}
Moreover taking $v=u_{n}$ one deduces easily as in the proof of Theorem
\ref{mcThm2} that $\left\Vert u_{n}\right\Vert \leq\left\Vert F\right\Vert
_{\ast}/\alpha$. Thus, up to a subsequence, one has for some $u\in H_{0}%
^{1}\left(  \Omega\right)  $%
\[
u_{n}\rightharpoonup u\quad\text{in }H_{0}^{1}\left(  \Omega\right)  ,\quad
u_{n}\rightarrow u\quad\text{in }L^{2}\left(  \Omega\right)  .
\]
This allows us to pass to the limit in (\ref{mc15}) and to conclude that $u$
satisfies%
\[
\int_{\Omega} \left\langle \mathbb{A}\left(  u\right)  \nabla u,\nabla
v\right\rangle +\lambda\left(  u\right)  uv =F\left(  v\right)  \quad\forall
v\in S_{n}\quad\forall n\in\mathbb{N}.
\]
By density of $\left(  S_{n}\right)  _{n}$ in $H_{0}^{1}\left(  \Omega\right)
$ the existence of a solution to (\ref{mc14}) follows.%
\endproof

\begin{remark}
Under the assumption (\ref{mc11}), (\ref{mc12}), $\lambda=0$, the problem
(\ref{mc14}) admits a unique solution. The proof of uniqueness is identical to
the proof of Theorem \ref{mcThm3}.
\end{remark}

\section{Numerical solution\label{SecNumSol}}

Suppose we are interested to find an approximate solution to (\ref{mc3}). If
we solve (\ref{mc4}) the problem reduces to find the corresponding solution
$u_{\mu}$ to (\ref{mc3.6}), i.e., to solve numerically a linear Poisson equation.

Problem (\ref{mc4}) is a fixed point equation for the function $G\left(
x\right)  :=\ell\left(  u_{x}\right)  $. Then a suitable algorithm is to
consider the sequence $\left(  x^{n}\right)  _{n}$ in $\mathbb{R}$ defined as%
\begin{equation}
\left\{
\begin{array}
[c]{ll}%
x^{0} & \text{starting guess,}\\
x^{n+1}=G\left(  x^{n}\right)  & n=0,1,\ldots
\end{array}
\right.  \label{mc16}%
\end{equation}
Let $\mu_{0}$ be a fixed point for $G$ and assume that the situation is as
illustrated in Figure \ref{Fig1},%
\begin{figure}[ptb]%
\centering
\includegraphics[
height=3.6737in,
width=3.6737in
]%
{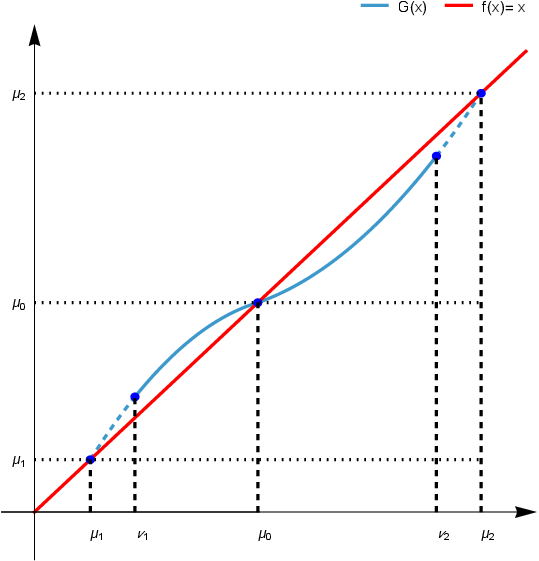}%
\caption{Illustration of the left- and right-hand side of the fixed point
equation $x=G\left(  x\right)  $ under the assumptions stated in
(\ref{mc17}).}%
\label{Fig1}%
\end{figure}
i.e., we suppose that, for some $\nu_{1}<\mu_{0}<\nu_{2}:$%
\begin{equation}
x<G\left(  x\right)  <\mu_{0}\quad\forall x\in\left]  \nu_{1},\mu_{0}\right[
\quad\text{and\quad}\mu_{0}<G\left(  x\right)  <x\quad\forall x\in\left]
\mu_{0},\nu_{2}\right[  . \label{mc17}%
\end{equation}
Then, for any $x^{0}\in\left]  \nu_{1},\nu_{2}\right[  $ the sequence
(\ref{mc16}) converges toward $\mu_{0}$. Indeed, suppose for instance that%
\[
x^{0}\in\left]  \nu_{1},\mu_{0}\right[  \text{,\quad then\quad}\nu_{1}%
<x^{0}<G\left(  x^{0}\right)  =x^{1}<\mu_{0},\quad\text{i.e.,\quad}x^{1}%
\in\left]  \nu_{1},\mu_{0}\right[  .
\]
For an induction argument, one supposes $x^{n}\in\left]  \nu_{1},\mu
_{0}\right[  $ and obtains%
\[
\nu_{1}<x^{n}<G\left(  x^{n}\right)  =x^{n+1}<\mu_{0}.
\]
Clearly, $\left(  x^{n}\right)  _{n}$ is an increasing sequence which can only
converge toward $\mu_{0}$. With a similar argument it follows that for
$x^{0}\in\left]  \mu_{0},\nu_{2}\right[  $, $\left(  x^{n}\right)  _{n}$ is a
decreasing sequence toward $\mu_{0}$.

If one imagines that $G$ has two other fixed points $\mu_{1}$, $\mu_{2}$ such
that $\mu_{1}<\mu_{0}<\mu_{2}$, one sees that the assumption $G\left(
x\right)  >x$ on the left of $x^{0}$ is necessary to have convergence of
$\left(  x^{n}\right)  _{n}$. Indeed if $x^{0}$ is on the left of $\mu_{2}$,
then $\left(  x^{n}\right)  _{n}$ does not converge toward $\mu_{2}$ but
toward $\mu_{0}$. Similarly, the conditions $G\left(  x\right)  <\mu_{0}$ on
the left of $\mu_{0}$ and $G\left(  x\right)  >\mu_{0}$ on the right of
$\mu_{0}$ are necessary to have convergence of the sequence $\left(
x^{n}\right)  _{n}$. Indeed, consider a function $G$ defined on $\left]
\nu_{1},\nu_{2}\right[  $ and such that for a point $x^{0}\in\left]  \nu
_{1},\mu_{0}\right[  $ one has $x^{1}=G\left(  x^{0}\right)  \in\left]
\mu_{0},\nu_{2}\right[  $ and $G\left(  x^{1}\right)  =x^{0}<\mu_{0}$. It is
easy to construct a function $G$ satisfying these conditions. This case is
illustrated in Figure \ref{Fig2}.
\begin{figure}[ptb]%
\centering
\includegraphics[
height=3.851in,
width=3.6668in
]%
{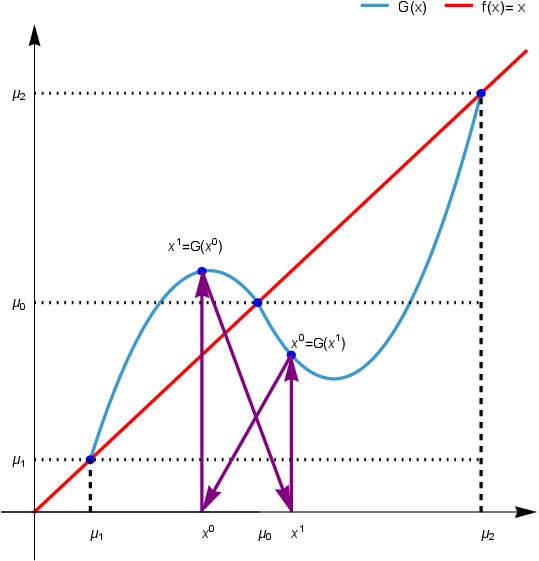}%
\caption[ ]{ }%
\label{Fig2}%
\end{figure}
Then, clearly the sequence (\ref{mc16}) is going to oscillate between $x^{0}$
and $x^{1}$ without converging.

Moreover, if one has%
\begin{equation}
x<G\left(  x\right)  <\mu_{2}\text{\quad on\ }\left]  \mu_{1},\mu_{0}\right[
\quad\text{and\quad}\mu_{1}<G\left(  x\right)  <x\quad\text{on }\left]
\mu_{0},\mu_{2}\right[  \label{mc18}%
\end{equation}
then the stability of the sequence (\ref{mc16}) is insured in the sense that%
\[
x^{0}\in\left]  \mu_{1},\mu_{2}\right[  \implies x^{n}\in\left]  \mu_{1}%
,\mu_{2}\right[  .
\]
Indeed, assuming $x^{0}\in\left]  \mu_{1},\mu_{0}\right[  $ one has by
(\ref{mc18})%
\[
\mu_{1}<x^{0}<G\left(  x^{0}\right)  <\mu_{2}\quad\text{i.e., }x^{1}\in\left]
\mu_{1},\mu_{2}\right[
\]
and if $x^{0}\in\left]  \mu_{0},\mu_{2}\right[  $%
\[
\mu_{1}<G\left(  x^{0}\right)  <x^{0}<\mu_{2}\quad\text{i.e., }x^{1}\in\left]
\mu_{1},\mu_{2}\right[  .
\]
By induction it is easy to show that $x^{n}\in\left]  \mu_{1},\mu_{2}\right[
$ for all $n$. This stability might be lost if%
\begin{equation}
\text{for some }\nu_{1}\in\left]  \mu_{1},\mu_{0}\right[  \qquad G\left(
\nu_{1}\right)  >\mu_{2}, \label{stas_instab}%
\end{equation}
i.e., if the range of $G$ leaves the interval $\left]  \mu_{1},\mu_{2}\right[
$ (see Figure \ref{Fig3}).%
\begin{figure}[ptb]%
\centering
\includegraphics[
height=3.851in,
width=3.6668in
]%
{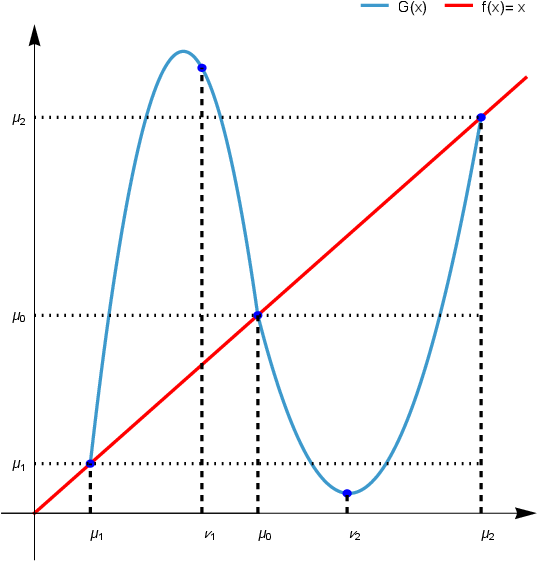}%
\caption[ ]{ }%
\label{Fig3}%
\end{figure}

To implement the solution method one can proceed in the following way. Let us
denote by $S$ a finite dimensional subspace of $H_{0}^{1}\left(
\Omega\right)  $. Suppose first that we are interested in solving (\ref{mc3})
or equivalently (\ref{mc4}). Let $u_{S,\mu}$ denote the solution to the
Galerkin discretization of (\ref{mc3.6})%
\begin{equation}
\text{find }u_{S,\mu}\in S\text{ such that}\quad%
{\displaystyle\int_{\Omega}}
\left\langle \mathbb{A}\left(  \cdot,\mu\right)  \nabla u_{S,\mu},\nabla
v\right\rangle +\lambda\left(  \cdot,\mu\right)  u_{S,\mu}v=F\left(  v\right)
\quad\forall v\in S. \label{mc19}%
\end{equation}
The related nonlinear discrete problem is given by%
\begin{equation}
\text{find }u_{S}\in S\text{ such that}\quad%
{\displaystyle\int_{\Omega}}
\left\langle \mathbb{A}\left(  \cdot,\ell\left(  u_{S}\right)  \right)  \nabla
u_{S},\nabla v\right\rangle +\lambda\left(  \cdot,\ell\left(  u_{S}\right)
\right)  u_{S}v=F\left(  v\right)  \quad\forall v\in S. \label{mc19nonlin}%
\end{equation}
Then, the discrete version of the fixed point equation in $\mathbb{R}$ is
given by%
\begin{equation}
\text{find }\mu_{S}\in\mathbb{R\quad}\text{such that\quad}\mu_{S}=G_{S}\left(
\mu_{S}\right)  \quad\text{with\quad}G_{S}\left(  x\right)  :=\ell\left(
u_{S,x}\right)  \label{defmueS}%
\end{equation}
and the corresponding fixed point iteration is%
\begin{equation}
\left\{
\begin{array}
[c]{ll}%
x_{S}^{0} & \text{starting guess,}\\
x_{S}^{n+1}=G_{S}\left(  x_{S}^{n}\right)  & n=0,1,\ldots
\end{array}
\right.  \label{mc16b}%
\end{equation}
With the shorthand $u_{S}^{n}:=u_{S,x_{S}^{n}}$ this iteration can be
reformulated by choosing some $u_{S}^{0}\in S$ as the starting guess and then
solving%
\[
\text{find }u_{S}^{n+1}\in S\text{ such that}\quad%
{\displaystyle\int_{\Omega}}
\left\langle \mathbb{A}\left(  \cdot,\ell\left(  u_{S}^{n}\right)  \right)
\nabla u_{S}^{n+1},\nabla v\right\rangle +\lambda\left(  \cdot,\ell\left(
u_{S}^{n}\right)  \right)  u_{S}^{n+1}v=F\left(  v\right)  \quad\forall v\in
S.
\]
This iteration again simplifies for coefficients $\mathbb{A}$, $\lambda$ as in
(\ref{mc5}), (\ref{mc6}). The linear Poisson problem
\begin{equation}
\text{find }\psi_{S}\in S\text{ such that}\quad%
{\displaystyle\int_{\Omega}}
\left\langle \nabla\psi_{S},\nabla v\right\rangle =F\left(  v\right)
\quad\forall v\in S, \label{discretePoisson}%
\end{equation}
delivers $\Psi_{S}$. After having found an approximation $x_{S}^{n}$ of the
solution to the fixed point equation in $\mathbb{R}:$%
\begin{equation}
x=\frac{\ell\left(  \psi_{S}\right)  }{a\left(  x\right)  ^{p}}, \label{mc20}%
\end{equation}
an approximation of (\ref{mc3}) is given in this case by%
\[
u_{S}^{n}=\psi_{S}/a\left(  x_{S}^{n}\right)  .
\]

\begin{remark}
The reasoning concerning the convergence of the fixed point iteration
(\ref{mc16}) in the beginning of this section could be applied verbatim to its
discrete version (\ref{mc16b}). For the simplified situation in (\ref{mc20}),
the verification of the sign conditions, e.g., in (\ref{mc17}) is feasible.
\end{remark}

\begin{remark}
If $a\left(  \cdot\right)  $ satisfies%
\[
0<\alpha\leq a\left(  r\right)  \leq\beta\quad\forall r\in\mathbb{R}%
\]
then, (\ref{mc20}) admits always at least one solution.
\end{remark}

\section{Error analysis}

We suppose in this section that we are in the framework of section 2. In
particular there exists a unique solution $u=u_{\mu}$ to
\begin{equation}%
\begin{cases}
u_{\mu}\in H_{0}^{1}(\Omega),\cr\int_{\Omega}\left\langle \mathbb{A}\left(
\cdot,\mu\right)  \nabla u_{\mu},\nabla v\right\rangle +\lambda\left(
\cdot,\mu\right)  u_{\mu}v=F\left(  v\right)  \quad\forall v\in H_{0}%
^{1}\left(  \Omega\right)  .\cr
\end{cases}
\label{5.1}%
\end{equation}
We denote by $\mathcal{L}_{\mu}^{-1}$ the continuous solution operator from
$H^{-1}(\Omega)$ into $H_{0}^{1}(\Omega)$ such that
\[
\mathcal{L}_{\mu}^{-1}F=u_{\mu}.
\]
Similarly for any finite dimensional subspace $S$ of $H_{0}^{1}(\Omega)$ there
exists a unique $u_{S,\mu}$ such that
\begin{equation}%
\begin{cases}
u_{S,\mu}\in S,\cr\int_{\Omega}\left\langle \mathbb{A}\left(  \cdot
,\mu\right)  \nabla u_{S,\mu},\nabla v\right\rangle +\lambda\left(  \cdot
,\mu\right)  u_{S,\mu}v=F\left(  v\right)  \quad\forall v\in S.\cr
\end{cases}
\label{5.2}%
\end{equation}
Next we analyse the discretization error. We denote by $C_{\operatorname*{P}}$
the Poincar\'{e}-Friedrichs constant in $\left\Vert u\right\Vert
_{L^{2}\left(  \Omega\right)  }\leq C_{\operatorname*{P}}\left\Vert
u\right\Vert $ for all $u\in H_{0}^{1}\left(  \Omega\right)  $. Let
$\left\Vert u\right\Vert _{H^{1}\left(  \Omega\right)  }:=\left(  \left\Vert
u\right\Vert ^{2}+\left\Vert u\right\Vert _{L^{2}\left(  \Omega\right)  }%
^{2}\right)  ^{1/2}$ and note that $\left\Vert u\right\Vert \leq\left\Vert
u\right\Vert _{H^{1}\left(  \Omega\right)  }\leq\left(  1+C_{\operatorname*{P}%
}^{2}\right)  ^{1/2}\left\Vert u\right\Vert $. We need a further assumption on
$\ell:H_{0}^{1}\left(  \Omega\right)  \rightarrow\mathbb{R}$ and on the
coefficients $\mathbb{A}$ and $\lambda$.

\begin{assumption}
\label{AssumpLocLip}The mappings $\mathbb{A}:\Omega\times\mathbb{R}%
\rightarrow\mathbb{R}^{n\times n}$, $\lambda:\Omega\times\mathbb{R}%
\rightarrow\mathbb{R}$ as in (\ref{mc1})-(\ref{mc2.6}) and $\ell:H_{0}%
^{1}\left(  \Omega\right)  \rightarrow\mathbb{R}$ are locally Lipschitz
continuous, more precisely, there exist constants $C_{\mathbb{A}}\left(
R\right)  $, $C_{\lambda}\left(  R\right)  $, $C_{\ell}\left(  R\right)  $
depending on $R>0$ such that%
\begin{equation}
\label{LipA}\begin{aligned} \left\Vert \mathbb{A}\left( \mathbf{x},\nu\right) -\mathbb{A}\left( \mathbf{x},\tilde{\nu}\right) \right\Vert _{2} & \leq C_{\mathbb{A}}\left( R\right) \left\vert \nu-\tilde{\nu}\right\vert \quad\text{a.e. }\mathbf{x} \in\Omega,\text{\quad}\\ &~~~~~~~~~~~~~~~~~~~~~~~~\forall\nu,\tilde{\nu}\in I_{R}:=\left\{ \mu \in\mathbb{R}\mid\left\vert \mu\right\vert \leq R\right\} ,\\ \left\vert \lambda\left( \mathbf{x},\nu\right) -\lambda\left( \mathbf{x},\tilde{\nu}\right) \right\vert & \leq C_{\lambda}\left( R\right) \left\vert \nu-\tilde{\nu}\right\vert \quad\text{a.e. }\mathbf{x}\in\Omega,\text{\quad}\forall\nu,\tilde{\nu}\in I_{R},\\ \left\vert \ell\left( w\right) -\ell\left( \tilde{w}\right) \right\vert & \leq C_{\ell}\left( R\right) \left\Vert w-\tilde{w}\right\Vert \quad\forall w,\tilde{w}\in B_{R}:=\left\{ v\in H_{0}^{1}\left( \Omega\right) \mid\left\Vert v\right\Vert \leq R\right\}. \end{aligned}
\end{equation}

\end{assumption}

C\'{e}a's lemma (see, e.g., \cite[Thm. 2.4.1]{CiarletPb}, \cite[(2.8.1)
Theorem]{scottbrenner3}, \cite[Thm. 32.2]{ErnGuermondII}) implies the
quasi-optimal estimate of the error between the solutions $u_{\mu}$ and
$u_{S,\mu}$ of (\ref{5.1}) and (\ref{5.2}), respectively:%
\begin{equation}
\left\Vert u_{\mu}-u_{S,\mu}\right\Vert \leq\frac{\beta}{\alpha}\eta_{\mu
,S}\left(  F\right)  \left\Vert F\right\Vert _{\ast}, \label{GalApprox}%
\end{equation}
where the approximation property $\eta_{\mu,S}\left(  F\right)  $ is defined
by $\eta_{\mu,S}\left(  0\right)  :=0$ and for $F\neq0$ by%
\begin{equation}
\eta_{\mu,S}\left(  F\right)  :=\frac{\inf_{w\in S}\left\Vert \mathcal{L}%
_{\mu}^{-1}F-w\right\Vert }{\left\Vert F\right\Vert _{\ast}}.
\label{defapproxprop}%
\end{equation}

\begin{lemma}
\label{Lem5.2}Let Assumption \ref{AssumpLocLip} be satisfied with the
constants $C_{\ell}$, $C_{\mathbb{A}}$, $C_{\lambda}$ therein. Let $F\in
H^{-1}\left(  \Omega\right)  $ and let $u_{\mu}$, $u_{S,\mu}$ denote the
solutions to (\ref{5.1}), (\ref{5.2}), respectively. Let $\mu,\mu_{S}$ denote
fixpoints of $\mu=\ell\left(  u_{\mu}\right)  $ and $\mu_{S}=\ell\left(
u_{S,\mu_{S}}\right)  $, respectively. Then,%
\begin{equation}
\begin{aligned} \left\vert \ell\left( u_{\mu}\right) -\ell\left( u_{S,\mu}\right) \right\vert \leq C_{\ell}\left( R_{F}\right) \left\Vert u_{\mu}-u_{S,\mu }\right\Vert & \text{~~with\quad}R_{F}:=\left\Vert F\right\Vert _{\ast}/\alpha,\\ \quad & \\ \left\Vert \mathbb{A}\left( \mathbf{x},\mu\right) -\mathbb{A}\left( \mathbf{x},\mu_{S}\right) \right\Vert _{2}\leq C_{\mathbb{A}}\left( D_{F}\right) \left\vert \mu-\mu_{S}\right\vert & \text{ ~~a.e. }\mathbf{x}\in\Omega\\ &\text{ with }D_{F}:=R_{F}C_{\ell}\left( R_{F}\right) +\left\vert \ell\left( 0\right) \right\vert ,\\ \quad & \\ \left\vert \lambda\left( \mathbf{x},\mu\right) -\lambda\left( \mathbf{x},\mu_{S}\right) \right\vert \leq C_{\lambda}\left( D_{F}\right) \left\vert \mu-\mu_{S}\right\vert & \text{~~a.e. }\mathbf{x}\in\Omega. \end{aligned} \label{LipEstAll}%
\end{equation}

\end{lemma}

%

\proof
By choosing $v=u_{\mu}$, $v=u_{S,\mu}$ in (\ref{5.1}), (\ref{5.2}),
respectively it follows%
\begin{equation}
\left\Vert u_{\mu}\right\Vert \leq\left\Vert F\right\Vert _{\ast}/\alpha
\quad\text{and\quad}\left\Vert u_{S,\mu}\right\Vert \leq\left\Vert
F\right\Vert _{\ast}/\alpha\label{enerfyestimates}%
\end{equation}
and the first claim is concluded from Assumption \ref{AssumpLocLip}.

The parameter $\mu$ is the solution of the fixed point equation $\mu
=\ell\left(  u_{\mu}\right)  $. Hence,%
\begin{equation}
\left\vert \mu\right\vert =\left\vert \ell\left(  u_{\mu}\right)  \right\vert
\leq\left\vert \ell\left(  u_{\mu}\right)  -\ell\left(  0\right)  \right\vert
+\left\vert \ell\left(  0\right)  \right\vert \leq C_{\ell}\left(
R_{F}\right)  \left\Vert u_{\mu}\right\Vert +\left\vert \ell\left(  0\right)
\right\vert \leq D_{F}. \label{estmodmue}%
\end{equation}
By similar arguments one obtains $\left\vert \mu_{S}\right\vert \leq D_{F}$.
Hence, Assumption \ref{AssumpLocLip} implies%
\begin{align*}
\left\Vert \mathbb{A}\left(  \mathbf{x},\mu\right)  -\mathbb{A}\left(
\mathbf{x},\mu_{S}\right)  \right\Vert _{2}  &  \leq C_{\mathbb{A}}\left(
D_{F}\right)  \left\vert \mu-\mu_{S}\right\vert \text{,}\\
\left\vert \lambda\left(  \mathbf{x},\mu\right)  -\lambda\left(
\mathbf{x},\mu_{S}\right)  \right\vert  &  \leq C_{\lambda}\left(
D_{F}\right)  \left\vert \mu-\mu_{S}\right\vert .
\end{align*}%
\endproof

\begin{theorem}
\label{ThmGalConv}Let Assumption \ref{AssumpLocLip} hold with the constants
$C_{\ell}$, $C_{\mathbb{A}}$, $C_{\lambda}$ therein. Let $R_{F}$, $D_{F}$ be
as in (\ref{LipEstAll}) and set%
\[
C_{\mathbb{A},\lambda}\left(  D_{F}\right)  :=\max\left\{  C_{\mathbb{A}%
}\left(  D_{F}\right)  ,C_{\lambda}\left(  D_{F}\right)  \right\}  .
\]
Suppose that the right-hand side $F\in H^{-1}\left(  \Omega\right)  $
satisfies%
\begin{equation}
C_{\mathbb{A},\lambda,\ell}^{F}\left\Vert F\right\Vert _{\ast}\leq
1\quad\text{with\quad}C_{\mathbb{A},\lambda,\ell}^{F}:=2C_{\ell}\left(
R_{F}\right)  \frac{C_{\mathbb{A},\lambda}\left(  D_{F}\right)  \left(
1+C_{\operatorname*{P}}^{2}\right)  }{\alpha^{2}}.
\label{smallnessconvergence}%
\end{equation}
Let $\mu$ be a fixed point of (\ref{mc4}) and $\mu_{S}$ one of (\ref{defmueS}%
). Then,
\begin{equation}
\left\vert \mu-\mu_{S}\right\vert \leq2C_{\ell}\left(  R_{F}\right)
\frac{\beta}{\alpha}\eta_{\mu,S}\left(  F\right)  \left\Vert F\right\Vert
_{\ast}. \label{Theofixedpointest}%
\end{equation}
The corresponding solutions $u_{\mu}$ and $u_{S,\mu_{S}}$ solve (\ref{5.1})
and (\ref{5.2}) and, in turn, the nonlinear problems (\ref{mc3}) and
(\ref{mc19nonlin}). It holds%
\begin{equation}
\left\Vert u_{\mu}-u_{S,\mu_{S}}\right\Vert \leq\frac{2}{\alpha}\eta_{\mu
,S}\left(  F\right)  \left\Vert F\right\Vert _{\ast}. \label{Theouest}%
\end{equation}

\end{theorem}

%

\proof
Let $\mu$ be a fixed point of (\ref{mc4}) and $\mu_{S}$ one of (\ref{defmueS}%
). With the constant $C_{\ell}$ as in (\ref{LipA}) we get%
\begin{equation}
\begin{aligned} \left\vert \mu-\mu_{S}\right\vert =\left\vert \ell\left( u_{\mu}\right) -\ell\left( u_{S,\mu_{S}}\right) \right\vert& \leq C_{\ell}\left( R_{F}\right) \left\Vert u_{\mu}-u_{S,\mu_{S}}\right\Vert \cr &\leq C_{\ell}\left( R_{F}\right) \left( \left\Vert u_{\mu}-u_{S,\mu}\right\Vert +\left\Vert u_{S,\mu}-u_{S,\mu_{S}}\right\Vert \right).\label{5.8} \end{aligned}
\end{equation}
The first term can be estimated by (\ref{GalApprox}):%
\begin{equation}
\left\Vert u_{\mu}-u_{S,\mu}\right\Vert \leq\frac{\beta}{\alpha}\eta_{\mu
,S}\left(  F\right)  \left\Vert F\right\Vert _{\ast}. \label{est1}%
\end{equation}
For $\rho\in\mathbb{R}$, let the bilinear form $a_{\rho}:H^{1}\left(
\Omega\right)  \times H^{1}\left(  \Omega\right)  \rightarrow\mathbb{R}$ be
given by%
\[
a_{\rho}\left(  v,w\right)  :=%
{\displaystyle\int_{\Omega}}
\left\langle \mathbb{A}\left(  \cdot,\rho\right)  \nabla v,\nabla
w\right\rangle +\lambda\left(  \cdot,\rho\right)  vw.
\]
For the last term in (\ref{5.8}) we set $d_{S}:=u_{S,\mu}-u_{S,\mu_{S}}$ and
obtain for any $w\in S:$%
\begin{align*}
0  &  =a_{\mu}\left(  u_{S,\mu},w\right)  -a_{\mu_{S}}\left(  u_{S,\mu_{S}%
},w\right) \\
&  =\int_{\Omega}\left\langle \left(  \mathbb{A}\left(  \cdot,\mu\right)
-\mathbb{A}\left(  \cdot,\mu_{S}\right)  \right)  \nabla u_{S,\mu},\nabla
w\right\rangle +\left(  \lambda\left(  \cdot,\mu\right)  -\lambda\left(
\cdot,\mu_{S}\right)  \right)  u_{S,\mu}w\\
&  +\int_{\Omega}\left\langle \mathbb{A}\left(  \cdot,\mu_{S}\right)  \nabla
d_{S},\nabla w\right\rangle +\lambda\left(  \cdot,\mu_{S}\right)  d_{S}w.
\end{align*}
We choose $w=d_{S}$ and employ the coercivity to get%
\begin{align}
\alpha\left\Vert d_{S}\right\Vert ^{2}  &  \leq a_{\mu_{S}}\left(  d_{S}%
,d_{S}\right) \nonumber\\
&  =-\int_{\Omega}\left\langle \left(  \mathbb{A}\left(  \cdot,\mu\right)
-\mathbb{A}\left(  \cdot,\mu_{S}\right)  \right)  \nabla u_{S,\mu},\nabla
d_{S}\right\rangle +\left(  \lambda\left(  \cdot,\mu\right)  -\lambda\left(
\cdot,\mu_{S}\right)  \right)  u_{S,\mu}d_{S}. \label{dsline2}%
\end{align}
We use the Lipschitz estimate (\ref{LipEstAll}) with $C_{\mathbb{A},\lambda
}\left(  D_{F}\right)  $ as in (\ref{smallnessconvergence}) to obtain for all
$v,w\in S:$%
\begin{align}
&  \left\vert \int_{\Omega}\left\langle \left(  \mathbb{A}\left(  \cdot
,\mu\right)  -\mathbb{A}\left(  \cdot,\mu_{S}\right)  \right)  \nabla v,\nabla
w\right\rangle +\left(  \lambda\left(  \cdot,\mu\right)  -\lambda\left(
\cdot,\mu_{S}\right)  \right)  vw\right\vert \nonumber\\
&  \qquad\leq C_{\mathbb{A},\lambda}\left(  D_{F}\right)  \left\Vert
v\right\Vert _{H^{1}\left(  \Omega\right)  }\left\Vert w\right\Vert
_{H^{1}\left(  \Omega\right)  }\left\vert \mu-\mu_{S}\right\vert \leq
C_{\mathbb{A},\lambda}\left(  D_{F}\right)  \left(  1+C_{\operatorname*{P}%
}^{2}\right)  \left\Vert v\right\Vert \left\Vert w\right\Vert \left\vert
\mu-\mu_{S}\right\vert . \label{adiffmue}%
\end{align}
An application of this to the right-hand side in (\ref{dsline2}) leads to%
\begin{equation}
\alpha\left\Vert d_{S}\right\Vert ^{2}\leq C_{\mathbb{A},\lambda}\left(
D_{F}\right)  \left(  1+C_{\operatorname*{P}}^{2}\right)  \left\Vert u_{S,\mu
}\right\Vert \left\Vert d_{S}\right\Vert \left\vert \mu-\mu_{S}\right\vert .
\label{dsfinalest}%
\end{equation}
The stability estimate $\left\Vert u_{S,\mu}\right\Vert \leq\left\Vert
F\right\Vert _{\ast}/\alpha$ implies%
\begin{equation}
\left\Vert d_{S}\right\Vert \leq\frac{C_{\mathbb{A},\lambda}\left(
D_{F}\right)  \left(  1+C_{\operatorname*{P}}^{2}\right)  }{\alpha^{2}%
}\left\Vert F\right\Vert _{\ast}\left\vert \mu-\mu_{S}\right\vert .
\label{est2}%
\end{equation}
The combination of (\ref{5.8}), (\ref{est1}), and (\ref{est2}) leads to%
\[
\left\vert \mu-\mu_{S}\right\vert \leq C_{\ell}\left(  R_{F}\right)
\frac{\beta}{\alpha}\eta_{\mu,S}\left(  F\right)  \left\Vert F\right\Vert
_{\ast}+C_{\ell}\left(  R_{F}\right)  \frac{C_{\mathbb{A},\lambda}\left(
D_{F}\right)  \left(  1+C_{\operatorname*{P}}^{2}\right)  }{\alpha^{2}%
}\left\Vert F\right\Vert _{\ast}\left\vert \mu-\mu_{S}\right\vert .
\]
The assumption on the smallness of $F$ stated in the theorem implies%
\[
\left\vert \mu-\mu_{S}\right\vert \leq2C_{\ell}\left(  R_{F}\right)
\frac{\beta}{\alpha}\eta_{\mu,S}\left(  F\right)  \left\Vert F\right\Vert
_{\ast}.
\]
The estimate of the difference $u_{\mu}-u_{S,\mu_{S}}$ follows from the
previous estimates:%
\begin{align}
\left\Vert u_{\mu}-u_{S,\mu_{S}}\right\Vert  &  \leq\left\Vert u_{\mu
}-u_{S,\mu}\right\Vert +\left\Vert u_{S,\mu}-u_{S,\mu_{S}}\right\Vert
\nonumber\\
&  \leq\frac{\beta}{\alpha}\eta_{\mu,S}\left(  F\right)  \left\Vert
F\right\Vert _{\ast}+\frac{C_{\mathbb{A},\lambda}\left(  D_{F}\right)  \left(
1+C_{\operatorname*{P}}^{2}\right)  }{\alpha^{2}}\left\Vert F\right\Vert
_{\ast}\left\vert \mu-\mu_{S}\right\vert \nonumber\\
&  \leq\left(  1+2C_{\ell}\left(  R_{F}\right)  \frac{C_{\mathbb{A},\lambda
}\left(  D_{F}\right)  \left(  1+C_{\operatorname*{P}}^{2}\right)  }%
{\alpha^{2}}\left\Vert F\right\Vert _{\ast}\right)  \frac{\beta}{\alpha}%
\eta_{\mu,S}\left(  F\right)  \left\Vert F\right\Vert _{\ast}\nonumber\\
&  \leq2\frac{\beta}{\alpha}\eta_{\mu,S}\left(  F\right)  \left\Vert
F\right\Vert _{\ast}. \label{umueminusumueSmueS}%
\end{align}
%

\endproof

If the mapping $\ell$ enjoys more regularity a superconvergence result for the
fixed point error can be proved.

\begin{corollary}
\label{Corsuper}Let the assumptions in Theorem \ref{ThmGalConv} be satisfied.
Assume, in addition, the mapping $\ell:H_{0}^{1}\left(  \Omega\right)
\rightarrow\mathbb{R}$ is locally in $C^{1,1}$, more precisely, there exist
adjusted constants $C_{\ell}\left(  R\right)  $, $D_{F}\left(  R\right)  $
such that%
\[%
\begin{array}
[c]{ll}%
\left\Vert \ell^{\prime}\left(  w\right)  \right\Vert _{\ast}\leq D_{\ell
}\left(  R\right)  & \forall w\in B_{R},\\
\left\vert \ell\left(  w\right)  -\left(  \ell\left(  w_{0}\right)
+\ell^{\prime}\left(  w_{0}\right)  \left(  w-w_{0}\right)  \right)
\right\vert \leq C_{\ell}\left(  R\right)  \left\Vert w-w_{0}\right\Vert ^{2}
& \forall w,w_{0}\in B_{R}.
\end{array}
\]
Then, the error of the discrete fixed point satisfies (with $R_{F}$ as in
(\ref{LipEstAll}) and $D:=\ell^{\prime}\left(  u_{\mu}\right)  \in
H^{-1}\left(  \Omega\right)  $)%
\[
\left\vert \mu-\mu_{S}\right\vert \leq C\eta_{\mu,S}\left(  F\right)  \left(
\eta_{\mu,S}\left(  F\right)  +\eta_{\mu,S}\left(  D\right)  \right)
\]
for a constant $C$ which depends exclusively on the data $\alpha,\beta
,F,\ell,\mathbb{A}$.
\end{corollary}

%

\proof
The generic constant $C$ in this proof depends only on the data $\alpha
,\beta,F,\ell,\mathbb{A}$ and may change its value from one occurrence to
another. A linearization of $\ell$ about $u_{\mu}$ yields%
\begin{equation}
\ell\left(  w\right)  =\ell\left(  u_{\mu}\right)  +D\left(  w-u_{\mu}\right)
+r\left(  w,u_{\mu}\right)  \quad\forall w\in B_{R} \label{llin}%
\end{equation}
with remainder%
\[
\left\vert r\left(  w,u_{\mu}\right)  \right\vert \leq C_{\ell}\left(
R\right)  \left\Vert w-u_{\mu}\right\Vert ^{2}.
\]
Hence,%
\begin{equation}
\mu-\mu_{S}=\ell\left(  u_{\mu}\right)  -\ell\left(  u_{S,\mu_{S}}\right)
=D\left(  u_{\mu}-u_{S,\mu_{S}}\right)  -r\left(  u_{S,\mu_{S}},u_{\mu
}\right)  \label{quantityD}%
\end{equation}
and $r\left(  \cdot,\cdot\right)  $ satisfies the estimate%
\[
\left\vert r\left(  u_{S,\mu_{S}},u_{\mu}\right)  \right\vert \leq C\left\Vert
u_{S,\mu_{S}}-u_{\mu}\right\Vert ^{2}\overset{\text{(\ref{umueminusumueSmueS}%
)}}{\leq}C\left(  \eta_{\mu,S}\left(  F\right)  \right)  ^{2}.
\]
For the linear part (\ref{llin}) we apply the Aubin-Nitsche argument (see,
e.g., \cite[Thm. 3.2.4]{CiarletPb}, \cite[Sect. 5.8]{scottbrenner3},
\cite[Sect. 32.3.1]{ErnGuermondII}) as follows. We consider the dual problem:
find $v_{\mu}\in H_{0}^{1}\left(  \Omega\right)  $ such that%
\begin{equation}
a_{\mu}\left(  w,v_{\mu}\right)  =D\left(  w\right)  \quad\forall w\in
H_{0}^{1}\left(  \Omega\right)  . \label{amueD}%
\end{equation}
Since $\mathbb{A}^{T}$ also satisfies the bounds (\ref{mc1}), problem
(\ref{amueD}) is well posed and $v_{\mu}=\left(  \mathcal{L}_{\mu}^{\ast
}\right)  ^{-1}D$ with the solution operator $\left(  \mathcal{L}_{\mu}^{\ast
}\right)  ^{-1}$ for (\ref{amueD}). The choice $w=v_{\mu}$ implies via the
coercivity of the bilinear form $\left\Vert v_{\mu}\right\Vert \leq\left\Vert
D\right\Vert _{\ast}/\alpha$. Let $e_{S}:=u_{S,\mu_{S}}-u_{\mu}$ and denote by
$w_{S}\in S$ the best approximation of $v_{\mu}:$%
\[
\inf_{g_{S}\in S}\left\Vert v_{\mu}-g_{S}\right\Vert =\left\Vert v_{\mu}%
-w_{S}\right\Vert
\]
which satisfies $\left\Vert w_{S}\right\Vert \leq\left\Vert v_{\mu}\right\Vert
$. It holds%
\begin{align*}
a_{\mu}\left(  e_{S},w_{S}\right)   &  =a_{\mu}\left(  u_{S,\mu_{S}}%
,w_{S}\right)  -a_{\mu_{S}}\left(  u_{S,\mu_{S}},w_{S}\right) \\
&  =\int_{\Omega}\left\langle \left(  \mathbb{A}\left(  \cdot,\mu\right)
-\mathbb{A}\left(  \cdot,\mu_{S}\right)  \right)  \nabla u_{S,\mu_{S}},\nabla
w_{S}\right\rangle +\left(  \lambda\left(  \cdot,\mu\right)  -\lambda\left(
\cdot,\mu_{S}\right)  \right)  u_{S,\mu_{S}}w_{S}.
\end{align*}
The same arguments as for (\ref{adiffmue}) lead to%
\[
\left\vert a_{\mu}\left(  e_{S},w_{S}\right)  \right\vert \leq C\left\Vert
e_{S}\right\Vert \left\Vert v_{\mu}\right\Vert \left\vert \mu-\mu
_{S}\right\vert \leq C\frac{\left\Vert D\right\Vert _{\ast}}{\alpha}\left\Vert
e_{S}\right\Vert \eta_{\mu,S}\left(  F\right)  .
\]
We combine these estimates and obtain for any $w_{S}$ the bound for the
quantity $D$ in (\ref{quantityD})\footnote{The adjoint approximation property
$\eta_{\mu,S}\left(  D\right)  $ has to be taken with respect to the operator
$\left(  \mathcal{L}_{\mu}^{\ast}\right)  ^{-1}$.}%
\begin{align*}
\left\vert D\left(  e_{S}\right)  \right\vert  &  =\left\vert a_{\mu}\left(
e_{S},v_{\mu}\right)  \right\vert \leq\left\vert a_{\mu}\left(  e_{S},v_{\mu
}-w_{S}\right)  \right\vert +\left\vert a_{\mu}\left(  e_{S},w_{S}\right)
\right\vert \\
&  \leq\beta\left\Vert e_{S}\right\Vert _{H^{1}\left(  \Omega\right)
}\left\Vert v_{\mu}-w_{S}\right\Vert _{H^{1}\left(  \Omega\right)
}+C\left\Vert e_{S}\right\Vert \eta_{\mu,S}\left(  F\right) \\
&  \leq C\left\Vert e_{S}\right\Vert \left(  \eta_{\mu,S}\left(  D\right)
+\eta_{\mu,S}\left(  F\right)  \right)  .
\end{align*}
Clearly, $\left\Vert D\right\Vert _{\ast}\leq D_{\ell}\left(  R_{F}\right)  $
and the assertion follows by combining these estimates with (\ref{Theouest}).%
\endproof

The following corollary provides an error estimate for the simplified scenario
(\ref{mc5}), (\ref{mc6}). We need the approximation property for the Poisson
problem (and $F\neq0$). Let $\mathcal{L}_{\Delta}^{-1}:H^{-1}\left(
\Omega\right)  \rightarrow H_{0}^{1}\left(  \Omega\right)  $ denote the
solution operator for the Poisson problem (\ref{linPoisson}) where $S=
H_{0}^{1}\left(  \Omega\right)  $. Then the approximation property for the
Poisson problem is defined by:%
\[
\eta_{\Delta,S}\left(  F\right)  :=\inf_{v\in S}\frac{\left\Vert
\mathcal{L}_{\Delta}^{-1}F-v\right\Vert }{\left\Vert F\right\Vert _{\ast}}.
\]

\begin{corollary}
\label{Corsimplvers}Let Assumption \ref{AssumpLocLip} be satisfied and assume
that (\ref{mc5}), (\ref{mc6}) hold. Let $F\in H^{-1}\left(  \Omega\right)  $
and set $C_{\ell}^{F}:=C_{\ell}\left(  \left\Vert F\right\Vert _{\ast}\right)
$. The right-hand side $F$ is assumed to be small enough such that%
\begin{equation}
\left\Vert F\right\Vert _{\ast}p\frac{\beta^{p-1}}{\alpha^{2p}}C_{\ell}%
^{F}C_{\mathbb{A}}\left(  \frac{C_{\ell}^{F}\left\Vert F\right\Vert _{\ast}%
}{\alpha^{p}}\right)  \leq\frac{1}{2}. \label{Fsmall}%
\end{equation}
Let $\mu$ be a fixed point of (\ref{mc4}) and $\mu_{S}$ one of (\ref{defmueS}%
). Then,
\[
\left\vert \mu-\mu_{S}\right\vert \leq2\frac{C_{\ell}^{F}}{\alpha^{p}}%
\eta_{\Delta,S}\left(  F\right)  \left\Vert F\right\Vert _{\ast}.
\]
The corresponding solutions $u_{\mu}$ of (\ref{5.1}) and $u_{S,\mu_{S}}$ of
(\ref{5.2}) solve the nonlinear problems (\ref{mc3}) and (\ref{mc19nonlin}).
The energy error can be estimated by%
\[
\left\Vert u_{\mu}-u_{S,\mu_{S}}\right\Vert \leq\eta_{\Delta,S}\left(
F\right)  \frac{\left\Vert F\right\Vert _{\ast}}{\alpha}\left(  1+\frac{1}%
{p}\left(  \frac{\alpha}{\beta}\right)  ^{p-1}\right)  .
\]

\end{corollary}

%

\proof
Let $\mu$ be a fixed point of (\ref{mc4}) and $\mu_{S}$ one of (\ref{defmueS}%
). Then%
\[
\left\vert \mu-\mu_{S}\right\vert =\left\vert \frac{\ell\left(  \mathcal{\psi
}\right)  }{a\left(  \mu\right)  ^{p}}-\frac{\ell\left(  \mathcal{\psi}%
_{S}\right)  }{a\left(  \mu_{S}\right)  ^{p}}\right\vert
\]
with $\psi=\mathcal{L}_{\Delta}^{-1}F$ and $\psi_{S}$ being its Galerkin
approximation in $S$. Then,
\[
\left\vert \frac{\ell\left(  \mathcal{\psi}\right)  }{a\left(  \mu\right)
^{p}}-\frac{\ell\left(  \mathcal{\psi}_{S}\right)  }{a\left(  \mu_{S}\right)
^{p}}\right\vert \leq\left\vert \frac{\ell\left(  \mathcal{\psi}\right)
-\ell\left(  \mathcal{\psi}_{S}\right)  }{a\left(  \mu\right)  ^{p}%
}\right\vert +\left\vert \ell\left(  \mathcal{\psi}_{S}\right)  \left(
\frac{1}{a\left(  \mu\right)  ^{p}}-\frac{1}{a\left(  \mu_{S}\right)  ^{p}%
}\right)  \right\vert .
\]
To estimate the difference of the first term we use the boundedness%
\[
\left\Vert \psi\right\Vert \leq\left\Vert F\right\Vert _{\ast}\quad
\text{and\quad}\left\Vert \psi_{S}\right\Vert \leq\left\Vert F\right\Vert
_{\ast}%
\]
and Assumption \ref{AssumpLocLip}%
\[
\left\vert \ell\left(  \mathcal{\psi}\right)  -\ell\left(  \mathcal{\psi}%
_{S}\right)  \right\vert \leq C_{\ell}^{F}\left\Vert \psi-\mathcal{\psi}%
_{S}\right\Vert \leq C_{\ell}^{F}\eta_{\Delta,S}\left(  F\right)  \left\Vert
F\right\Vert _{\ast}\text{.}%
\]
In this way we get%
\[
\left\vert \frac{\ell\left(  \mathcal{\psi}\right)  -\ell\left(
\mathcal{\psi}_{S}\right)  }{a\left(  \mu\right)  ^{p}}\right\vert \leq
\frac{C_{\ell}^{F}}{\alpha^{p}}\eta_{\Delta,S}\left(  F\right)  \left\Vert
F\right\Vert _{\ast}.
\]
The homogeneity of $\ell$ as in (\ref{mc6}) implies $\ell\left(  0\right)
=0$. Hence, for the second term we first employ the local Lipschitz continuity
of $\ell\left(  \cdot\right)  $ for%
\begin{equation}
\left\vert \ell\left(  \psi_{S}\right)  \right\vert =\left\vert \ell\left(
0\right)  \right\vert +\left\vert \ell\left(  \psi_{S}\right)  -\ell\left(
0\right)  \right\vert \leq C_{\ell}^{F}\left\Vert F\right\Vert _{\ast}
\label{estlpsiS}%
\end{equation}
and obtain%
\begin{align*}
\left\vert \ell\left(  \mathcal{\psi}_{S}\right)  \left(  \frac{1}{a\left(
\mu\right)  ^{p}}-\frac{1}{a\left(  \mu_{S}\right)  ^{p}}\right)  \right\vert
&  \leq\left\vert \ell\left(  \psi_{S}\right)  \right\vert \frac{\left\vert
a\left(  \mu_{S}\right)  ^{p}-a\left(  \mu\right)  ^{p}\right\vert }{a\left(
\mu\right)  ^{p}a\left(  \mu_{S}\right)  ^{p}}\\
&  \leq C_{\ell}^{F}\left\Vert F\right\Vert _{\ast}p\frac{\beta^{p-1}}%
{\alpha^{2p}}\left\vert a\left(  \mu_{S}\right)  -a\left(  \mu\right)
\right\vert .
\end{align*}
To estimate the last factor we use the local Lipschitz continuity of
$\mathbb{A}$ (cf. Assumption \ref{AssumpLocLip} ) which simplifies for
(\ref{mc5}), (\ref{mc6}) to%
\[
\left\vert \left(  a\left(  \mu\right)  -a\left(  \mu_{S}\right)  \right)
\right\vert \leq C_{\mathbb{A}}\left(  \rho\right)  \left\vert \mu-\mu
_{S}\right\vert \quad\text{with }\rho:=\max\left\{  \left\vert \mu\right\vert
,\left\vert \mu_{S}\right\vert \right\}  .
\]
The equation (\ref{FPsimplified}) for $\mu$ implies%
\[
\left\vert \mu\right\vert =\frac{\left\vert \ell\left(  \psi\right)
\right\vert }{\left\vert a\left(  \mu\right)  \right\vert ^{p}}\leq
\frac{\left\vert \ell\left(  \psi\right)  \right\vert }{\alpha^{p}}.
\]
The numerator can be estimated as in (\ref{estlpsiS}) by $\left\vert
\ell\left(  \psi\right)  \right\vert \leq C_{\ell}^{F}\left\Vert F\right\Vert
_{\ast}$ so that $\left\vert \mu\right\vert \leq C_{\ell}^{F}\left\Vert
F\right\Vert _{\ast}\alpha^{-p}$. The same arguments applied to the discrete
fixed point yield $\left\vert \mu_{S}\right\vert \leq C_{\ell}^{F}\left\Vert
F\right\Vert _{\ast}\alpha^{-p}$. Hence,%
\[
\left\vert a\left(  \mu\right)  -a\left(  \mu_{S}\right)  \right\vert \leq
C_{\mathbb{A}}\left(  \frac{C_{\ell}^{F}\left\Vert F\right\Vert _{\ast}%
}{\alpha^{p}}\right)  \left\vert \mu-\mu_{S}\right\vert .
\]
The combination of these inequalities leads to%
\[
\left\vert \mu-\mu_{S}\right\vert \leq\frac{C_{\ell}^{F}}{\alpha^{p}}%
\eta_{\Delta,S}\left(  F\right)  \left\Vert F\right\Vert _{\ast}+C_{\ell}%
^{F}\left\Vert F\right\Vert _{\ast}p\frac{\beta^{p-1}}{\alpha^{2p}%
}C_{\mathbb{A}}\left(  \frac{C_{\ell}^{F}\left\Vert F\right\Vert _{\ast}%
}{\alpha^{p}}\right)  \left\vert \mu-\mu_{S}\right\vert .
\]
From the smallness assumption on $F$ (cf. (\ref{Fsmall})) the estimate for
$\mu-\mu_{S}$ follows.

For the difference $u_{\mu}-u_{S,\mu_{S}}$ we employ the previous estimates
(with $\psi:=\mathcal{L}_{\Delta}^{-1}F$ and $\psi_{S}$ as in
(\ref{discretePoisson})) for%
\begin{align*}
\left\Vert u_{\mu}-u_{S,\mu_{S}}\right\Vert  &  \leq\left\Vert \frac{\psi
}{a\left(  \mu\right)  }-\frac{\psi_{S}}{a\left(  \mu_{S}\right)  }\right\Vert
\leq\left\Vert \frac{\psi-\psi_{S}}{a\left(  \mu\right)  }\right\Vert
+\left\Vert \psi_{S}\left(  \frac{1}{a\left(  \mu\right)  }-\frac{1}{a\left(
\mu_{S}\right)  }\right)  \right\Vert \\
&  \leq\eta_{\Delta,S}\left(  F\right)  \frac{\left\Vert F\right\Vert _{\ast}%
}{\alpha}+\left\Vert F\right\Vert _{\ast}\frac{1}{\alpha^{2}}C_{\mathbb{A}%
}\left(  \frac{C_{\ell}^{F}\left\Vert F\right\Vert _{\ast}}{\alpha^{p}%
}\right)  \left\vert \mu-\mu_{S}\right\vert \\
&  \leq\eta_{\Delta,S}\left(  F\right)  \frac{\left\Vert F\right\Vert _{\ast}%
}{\alpha}\left(  1+\frac{1}{p}\left(  \frac{\alpha}{\beta}\right)
^{p-1}\right)  .
\end{align*}%
\endproof

\begin{remark}
A superconvergence result for the fixed point error for the simplified
scenario (\ref{mc5}), (\ref{mc6}) can be derived along the same lines of
arguments as in the proof of Corollary \ref{Corsuper}; the details are omitted.
\end{remark}

\section{Numerical experiments}

In this section the results of some numerical experiments are presented which
illustrate some characteristic behaviour of the fully discrete numerical
method. We restrict to the one-dimensional domain $\Omega:=\left(  0,1\right)
$, the right-hand side $f=1$, and%
\begin{equation}
F\left(  v\right)  =\int_{\Omega}v,\quad\lambda=0\text{,\quad}\mathbb{A}%
\left(  \mathbf{x},r\right)  :=a\left(  x,r\right)  \operatorname*{Id}
\label{choiceF}%
\end{equation}
for some coefficient $a\left(  \cdot,\cdot\right)  $ at our disposition. We
employ a Galerkin discretization with continuous, piecewise linear spline
functions on a uniform partitioning of $\Omega$ with mesh points $x_{i}=ih$,
$0\leq i\leq n+1$, and mesh width $h=1/\left(  n+1\right)  $. The intervals
are denoted by $\tau_{i}:=\left(  x_{i-1},x_{i}\right)  $, $1\leq i\leq N+1$,
and the mesh by $\mathcal{T}_{h}:=\left\{  \tau_{i}:1\leq i\leq N+1\right\}
$. The space of spline functions is given by%
\begin{equation}
S_{h}:=\left\{  u\in H_{0}^{1}\left(  \Omega\right)  \mid\forall\tau
\in\mathcal{T}_{h}\quad\left.  u\right\vert _{\tau}\in\mathbb{P}_{1}\left(
\tau\right)  \right\}  , \label{defSh}%
\end{equation}
where $\mathbb{P}_{1}\left(  \tau\right)  $ is the space of polynomials of
maximal degree $\leq1$ on $\tau$. We write $\mu_{h}$, $u_{h}$ short for
$\mu_{S_{h}}$ and $u_{S_{h}}$ and similarly for other quantities.

\subsection{Convergence of the fixed point iteration for (\ref{mc20}%
)\label{SubSecConvFPIT}}

We consider the simplified scenario as in Corollary \ref{CorSimplCoeff},
(\ref{choiceF}), and (\ref{defSh}). Let%
\[
\mu_{1}=\frac{1}{4},\quad\mu_{0}=1,\quad\mu_{2}=2,\quad\nu_{1}=\frac{1}%
{4}+\frac{1}{5},\quad\nu_{2}=2-\frac{1}{5}.
\]
For the mapping $\ell\left(  \cdot\right)  $ we choose%
\begin{equation}
\ell\left(  u\right)  :=12\int_{\Omega}\left\vert u^{\prime}\right\vert ^{2}.
\label{deflj}%
\end{equation}

\subsubsection{The convergent case (\ref{mc17})\label{SecGoodCase}}

In this section, our parameter choice corresponds to the condition
(\ref{mc17}) so that the fixed point iteration converges for any starting
guess in $\left]  \nu_{1},\nu_{2}\right[  $. Let $G$ be given by%
\begin{subequations}
\label{numexp_defG}
\end{subequations}%
\begin{equation}
G\left(  \mu\right)  :=\left\{
\begin{array}
[c]{cc}%
G_{\operatorname*{L}}\left(  \mu\right)  & \text{for }\mu\in\left[  \mu
_{1},\mu_{0}\right]  ,\\
G_{\operatorname*{R}}\left(  \mu\right)  & \text{for }\mu\in\left[  \mu
_{0},\mu_{2}\right]  ,
\end{array}
\right.  \tag{%
\ref{numexp_defG}%
a}\label{numexp_defGa}%
\end{equation}
where $G_{\operatorname*{L}}$, $G_{\operatorname*{R}}$ are the quadratic
polynomials which interpolates the values $\left(  \mu_{1},\mu_{1}\right)  $,
$\left(  \nu_{1},\nu_{1}+\frac{1}{10}\right)  $, $\left(  \mu_{0},\mu
_{0}\right)  $ and $\left(  \mu_{0},\mu_{0}\right)  ,\left(  \nu_{2},\nu
_{2}-\frac{1}{10}\right)  $, $\left(  \mu_{2},\mu_{2}\right)  $. They are
explicitly given by (this is the case depicted in Fig. \ref{Fig1})%
\begin{equation}
G_{\operatorname*{L}}\left(  x\right)  =\frac{1}{22}\left(  -20x^{2}%
+47x-5\right)  \quad\text{and\quad}G_{\operatorname*{R}}\left(  x\right)
=\frac{1}{8}\left(  5x^{2}-7x+10\right)  . \tag{%
\ref{numexp_defG}%
b}\label{numexp_defGb}%
\end{equation}
This function $G$ arises from (\ref{FPsimplified}) by choosing%
\begin{equation}
a\left(  \mu\right)  :=\left(  \frac{\ell\left(  \psi\right)  }{G\left(
\mu\right)  }\right)  ^{1/p}\quad\text{and\quad}\lambda=0.
\label{num_exp_defa}%
\end{equation}

The choice of $F$ and $\ell$ as in (\ref{choiceF}) and (\ref{deflj}) leads to
\[
\psi=\frac{x\left(  1-x\right)  }{2}\quad\text{and}\quad\ell\left(
\psi\right)  =1.
\]
The corresponding value of the coefficient $a\left(  \cdot\right)  $ is%
\begin{equation}
a\left(  \mu\right)  =G\left(  \mu\right)  ^{-1/2}, \label{defa12}%
\end{equation}
where $G$ is as in (\ref{numexp_defG}). The solution for the fixed point
$\mu_{0}=1$ of $\mu=G\left(  \mu\right)  $ is given by%
\[
u\left(  x\right)  =\frac{x\left(  1-x\right)  }{2}.
\]
The convergence plot and semi-log plot of the error $\left\vert \mu_{h}%
-\mu_{h,m}\right\vert $ in Figure \ref{Figgoodcase} illustrate the exponential
convergence of the fixed points iteration $\mu_{h,m+1}=G_{h}\left(  \mu
_{h,m}\right)  $ for the starting guess $\mu_{h,0}:=\nu_{1}$.
\begin{figure}[tbp] \centering
\begin{tabular}
[c]{lll}%
{\includegraphics[
height=2.1707in,
width=2.6212in
]%
{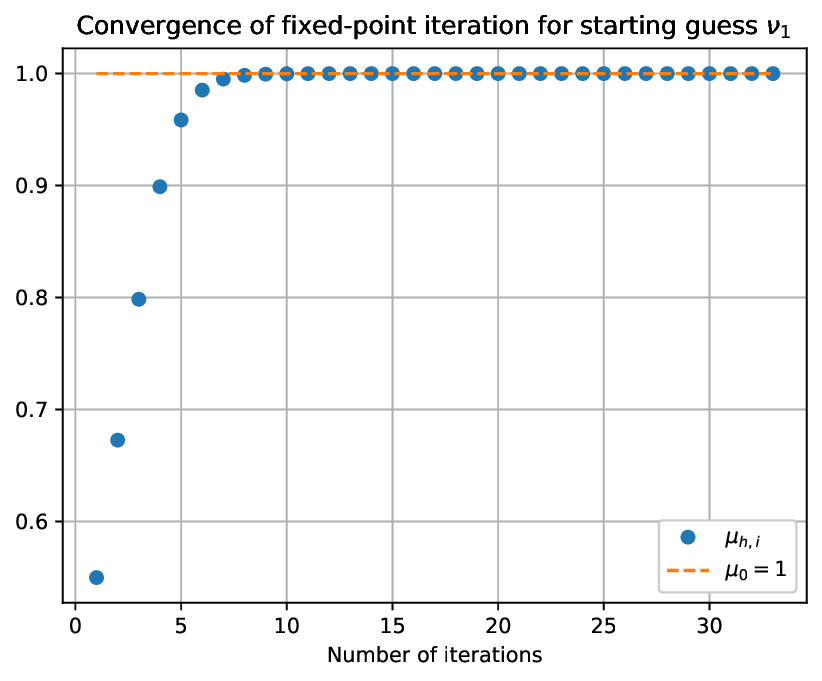}%
}
&  &
{\includegraphics[
height=2.1517in,
width=2.6783in
]%
{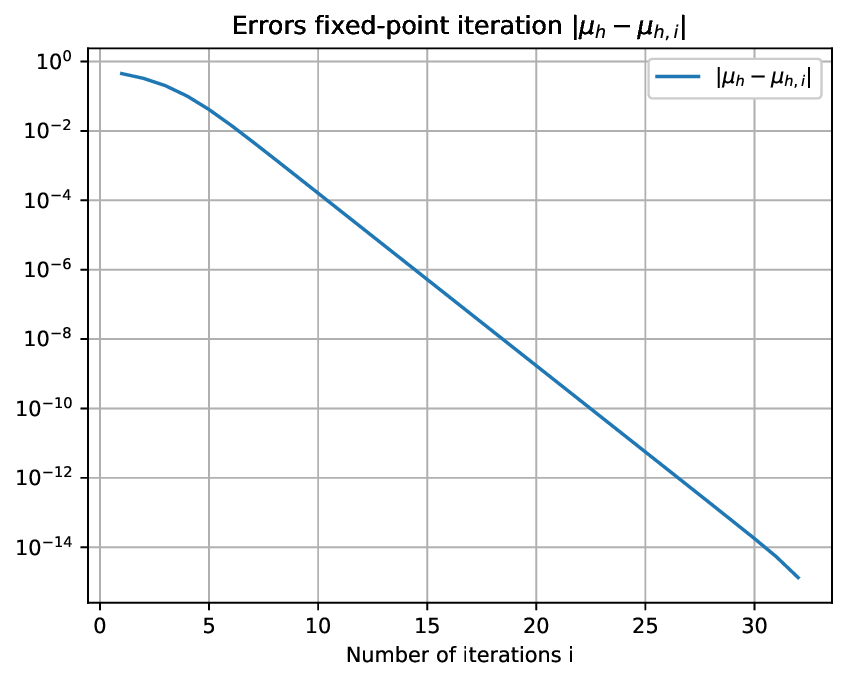}%
}
\end{tabular}
\caption{Convergence of the fixed point iteration
for a mesh with $h=2^{-13}$. Left: The iterations $\mu_{h,n}$ are converging
rapidly to the discrete fixed point $\mu_{h}$ which is very close to $\mu
_{0}=1$ for this fine mesh. Right: The semilog plot of the error illustrates
the exponential convergence of the fixed point iteration.\label{Figgoodcase}}%
\end{figure}%

As a reference value we have chosen $\mu_{h}\leftarrow\mu_{h,M}$ for some
$M\gg m$.\ Due to the fine mesh $h=2^{-13}$ the reference value $\mu_{h,M}$ is
very close to the exact value $\mu_{0}=1$. Convergence tests with respect to
$h$ will be reported in section \ref{SecNumExp_FDM}.

\subsubsection{The divergent, bounded case (\ref{mc18})}

Let $F$, $\ell$ be as in the previous section. We consider the parameter case
as depicted in Figure \ref{Fig2}. Let%
\begin{equation}
\mu_{1}=\frac{1}{4},\quad\mu_{0}=1,\quad\mu_{2}=2,\quad\nu_{1}=\frac{3}%
{4},\quad\nu_{2}=1+\frac{3}{20}. \label{stas_div_bound}%
\end{equation}
For this case, $G$ is defined by%
\begin{subequations}
\label{numexpbdG}
\end{subequations}%
\begin{equation}
G\left(  \mu\right)  :=\left\{
\begin{array}
[c]{cc}%
G_{\operatorname*{L}}\left(  \mu\right)  & \text{for }\mu\in\left[  \mu
_{1},\mu_{0}\right]  ,\\
G_{\operatorname*{R}}\left(  \mu\right)  & \text{for }\mu\in\left[  \mu
_{0},\mu_{2}\right]  ,
\end{array}
\right.  \tag{%
\ref{numexpbdG}%
a}\label{numexp_bdGa}%
\end{equation}
where $G_{\operatorname*{L}}$, $G_{\operatorname*{R}}$ are the quadratic
polynomials which interpolate the values $\left(  \mu_{1},\mu_{1}\right)  $,
$\left(  \nu_{1},\nu_{1}+\frac{2}{5}\right)  $, $\left(  \mu_{0},\mu
_{0}\right)  $ and $\left(  \mu_{0},\mu_{0}\right)  ,\left(  \nu_{2},\nu
_{2}-\frac{2}{5}\right)  $, $\left(  \mu_{2},\mu_{2}\right)  $. They are
explicitly given by (this is the case depicted in Fig. \ref{Fig2})%
\begin{equation}
G_{\operatorname*{L}}\left(  x\right)  =-\frac{1}{5}\left(  16x^{2}%
-25x+4\right)  \quad\text{and\quad}G_{\operatorname*{R}}\left(  x\right)
=\frac{1}{51}\left(  160x^{2}-429x+320\right)  . \tag{%
\ref{numexpbdG}%
b}\label{numexp_bdGb}%
\end{equation}
The function $\psi$ is as in the previous experiment, $\lambda=0$, and
$a\left(  \cdot\right)  $ as in (\ref{num_exp_defa}) for the function $G$
defined in (\ref{numexpbdG}). In this case, the fixed point iteration fails to
converge for any tested starting guess (even for the choice $\mu_{h,0}=1$ due
to roundoff errors); however the iteration stays bounded in all cases. These
observations are in full accordance with the reasoning for case (\ref{mc18})
depicted in Figure \ref{Fig2}. The convergence behaviour is illustrated in
Figure \ref{Fig_bounded_case} for the starting guess $\mu_{h,0}:=\nu_{1}$.
\begin{figure}[ptb]%
\centering
\includegraphics[
height=2.4552in,
width=3.0096in
]%
{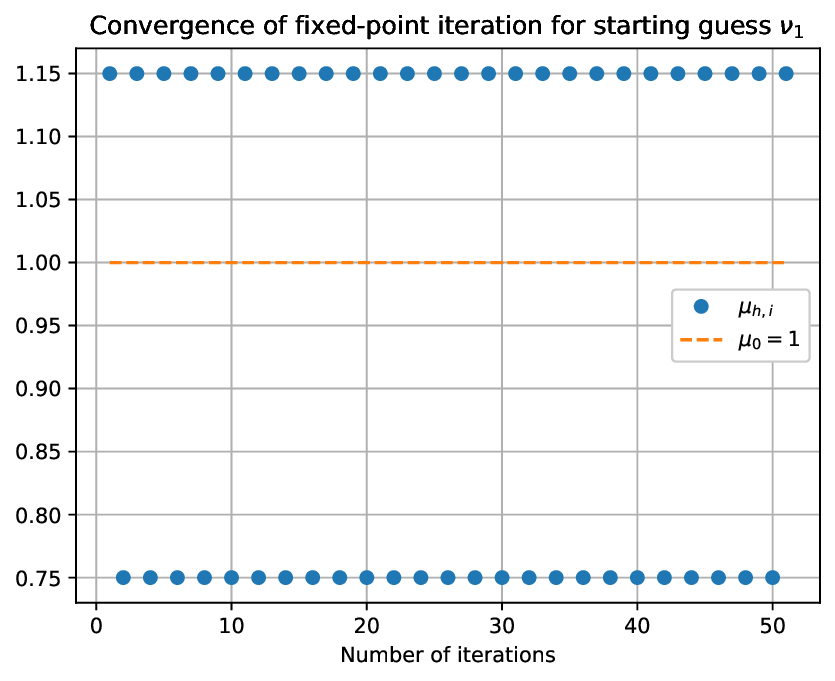}%
\caption{Divergent fixed point iteration for the parameter choice
(\ref{stas_div_bound}) and starting guess $\mu_{h,0}:=\nu_{1}$ such that the
iterates alternate between two values.}%
\label{Fig_bounded_case}%
\end{figure}

\subsubsection{The divergent, unbounded case (\ref{stas_instab})}

Let $F$, $\ell$ be as in the previous sections. We consider the parameter case
as depicted in Figure \ref{Fig3}. Let%
\begin{equation}
\mu_{1}=\frac{1}{4},\quad\mu_{0}=1,\quad\mu_{2}=2,\quad\nu_{1}=\frac{3}%
{4},\quad\nu_{2}=\frac{7}{5}. \label{stas_exploding}%
\end{equation}
For this case, $G$ is defined by%
\begin{subequations}
\label{numexp_expG}
\end{subequations}%
\begin{equation}
G\left(  \mu\right)  :=\left\{
\begin{array}
[c]{cc}%
G_{\operatorname{L}}\left(  \mu\right)  & \text{for }\mu\in\left[  \mu_{1}%
,\mu_{0}\right]  ,\\
G_{\operatorname{R}}\left(  \mu\right)  & \text{for }\mu\in\left[  \mu_{0}%
,\mu_{2}\right]  ,
\end{array}
\right.  \tag{%
\ref{numexp_expG}%
a}\label{numexp_expGa}%
\end{equation}
where
\begin{equation}
G_{\operatorname{L}}\left(  x\right)  =-3+16x-12x^{2}\quad\text{and\quad
}G_{\operatorname{R}}\left(  x\right)  =\frac{65}{6}-\frac{61}{4}x+\frac
{65}{12}x^{2} \tag{%
\ref{numexp_expG}%
b}\label{numexp_expGb}%
\end{equation}
are the quadratic polynomials which interpolate the values $\left(  \mu
_{1},\mu_{1}\right)  $, $\left(  \nu_{1},\nu_{1}+\frac{3}{2}\right)  $,
$\left(  \mu_{0},\mu_{0}\right)  $ and $\left(  \mu_{0},\mu_{0}\right)  $,
$\left(  \nu_{2},\nu_{2}-\frac{13}{10}\right)  $, $\left(  \mu_{2},\mu
_{2}\right)  $. This is the case depicted in Fig. \ref{Fig3}. The function
$\psi$ is as in the previous experiment, $\lambda=0$, and $a\left(
\cdot\right)  $ as in (\ref{num_exp_defa}) for the function $G$ defined in
(\ref{numexp_expG}). The fixed point iteration diverges toward infinity in all
cases. This is in accordance with the reasoning for case (\ref{stas_instab})
depicted in Figure \ref{Fig3}. The divergence is illustrated in Figure
\ref{Fig_exploding} for the starting guess $\mu_{h,0}:=\nu_{1}$.%
\begin{figure}[ptb]%
\centering
\includegraphics[
height=2.5512in,
width=3.1661in
]%
{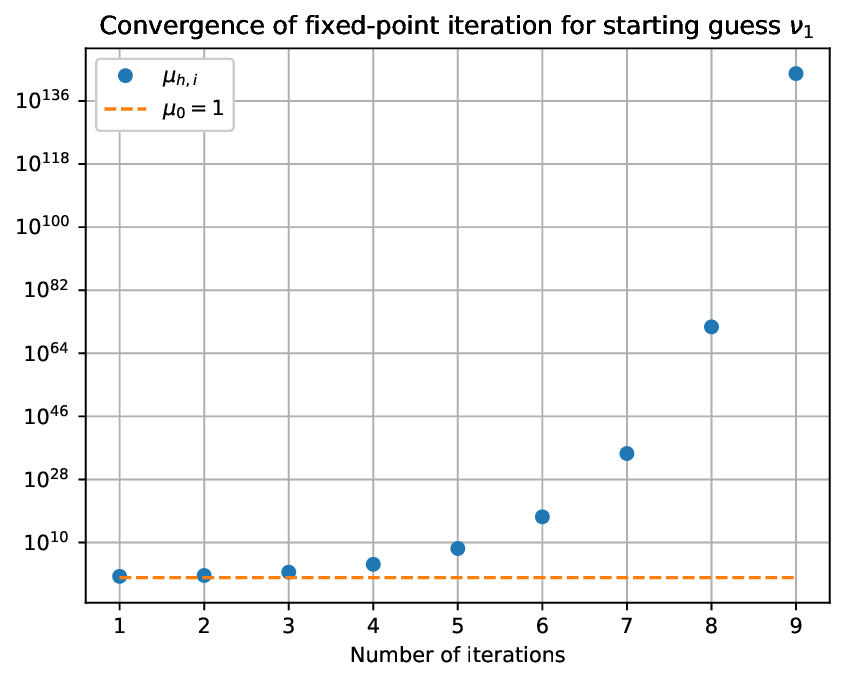}%
\caption{Divergent fixed point iteration for the parameter choice
(\ref{stas_exploding}) and starting guess $\mu_{h,0}:=\nu_{1}$ such that the
iterates diverges toward infinity.}%
\label{Fig_exploding}%
\end{figure}

\subsection{The general fixed point iteration}

Under the assumptions of Corollary \ref{CorSimplCoeff} the fixed point
iteration simplifies and the solution of a Poisson problem in each iteration
step is avoided: $G\left(  \mu\right)  =c/a\left(  \mu\right)  ^{p}$ for
$c=\ell\left(  \psi\right)  $. In this way, it becomes feasible to verify the
conditions on $G$ for the convergence of the fixed point iteration in a
neighbourhood of $\mu_{0}$ as described in (\ref{SecNumSol}). In this section
we consider a more general diffusion coefficient $\mathbb{A}\left(
x,\mu\right)  =a_{\varepsilon,m}\left(  x,\mu\right)  \operatorname*{Id}$ with
a parameter-dependent coefficient $a_{\varepsilon,m}\left(  \cdot
,\cdot\right)  :$%
\[
a_{\varepsilon,m}\left(  x,\mu\right)  :=a\left(  \mu\right)  \left(
1+\varepsilon\cos\left(  mx\right)  \right)  ,\qquad a\text{ as in
(\ref{defa12}),}%
\]
for $\varepsilon\in\left[  0,1\right[  $ and $m\in\mathbb{N}$. For
$\varepsilon=0$, the coefficient $a_{\varepsilon,m}\left(  \cdot,\cdot\right)
$ is the same as $a\left(  \cdot\right)  $ in first experiment of section
\ref{SubSecConvFPIT}. The function $G$ is given by $G\left(  \mu\right)
=\ell\left(  u_{\mu}\right)  $ and its discrete version by $G_{h}\left(
\mu\right)  =\ell\left(  u_{h,\mu}\right)  $. In Figure \ref{fig_perturb} this
function $G_{h}\left(  \mu\right)  $ is depicted for different values of
$\varepsilon$ and $m$.%
\begin{figure}[tbp] \centering
\begin{tabular}
[c]{lll}%
{\includegraphics[
height=2.2649in,
width=2.8374in
]%
{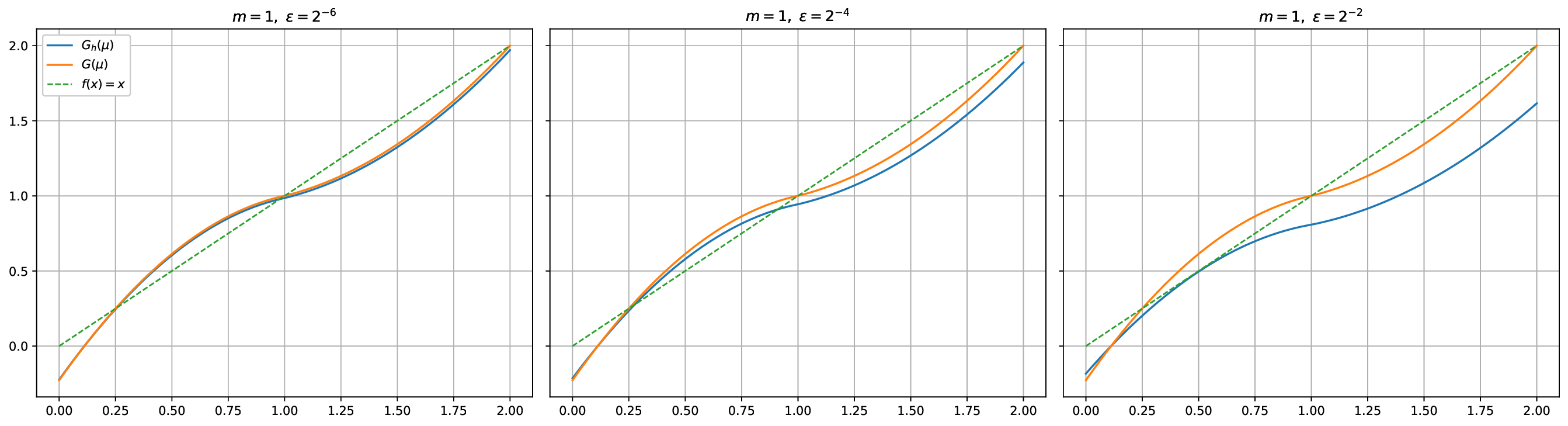}%
}
&  &
\raisebox{0.0069in}{\includegraphics[
height=2.2485in,
width=2.8167in
]%
{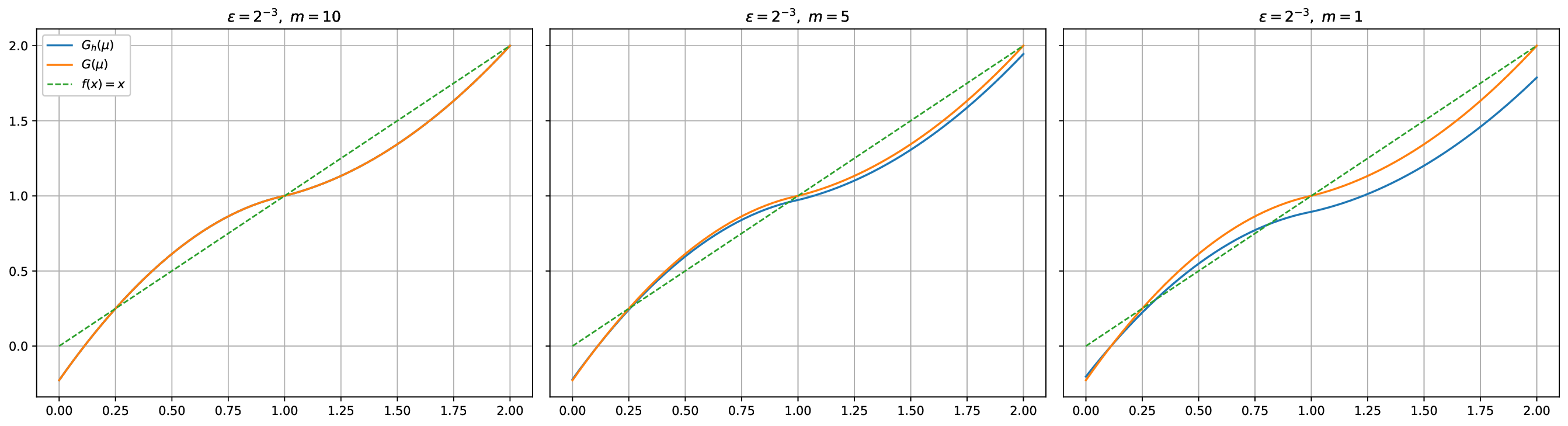}%
}
\end{tabular}
\caption{Plot of the function $G_{h}\left(  \mu\right)  $ for different choices of
$\varepsilon\in\left[  0,1\right[  $ and $m\in\mathbb{N}$. Left: $G_{h}\left(
\mu\right)  $ for $a_{\varepsilon,1}$ and $\varepsilon=10^{-2},10^{-4},10^{-6}$.
Right $G_{h}\left(  \mu\right)  $ for $a_{1/8,m}$ for $m=10, 5, 1$.\label{fig_perturb}}%
\end{figure}
We see that the functions $G_{h}$ for all choices of parameters $\left(
\varepsilon,m\right)  $ enjoy property (\ref{mc17}) which implies convergence
of the fixed point iteration. Larger values of $\varepsilon$ generate a small
shift of the original function ($\varepsilon=0$) while higher oscillations in
the perturbation ($m>1$) reduce this shift.

\subsection{Convergence of the fully discrete method\label{SecNumExp_FDM}}

Finally, we investigate the convergence order of the Galerkin/fixed point
approximation with respect to the mesh size in the finite element space
$S_{h}$ as defined in (\ref{defSh}). The choice of parameters are as in the
first experiment in section \ref{SubSecConvFPIT}. Since we are using piecewise
linear elements the approximation property (\ref{defapproxprop}) satisfies
$\eta_{\mu,h}\left(  F\right)  \leq Ch$. According to Theorem \ref{ThmGalConv}
we expect linear convergence and the solution of the original nonlinear
problem. The parameter set is chosen as for the convergent case described in
section \ref{SecGoodCase}. The fixed point iteration is stopped when the
machine tolerance has been reached.

In order to predict the convergence order of the fixed point with respect to
the mesh width $h$ we apply Corollary \ref{Corsuper}. Let $u_{\mu}$ denote the
exact solution of the Poisson problem with right-hand side $f=1$%
\[
\text{find }u_{\mu}\in H_{0}^{1}\left(  \Omega\right)  \ \text{s.t.\quad}%
\int_{\Omega}a\left(  \mu\right)  \nabla u_{\mu}\cdot\nabla v=\int_{\Omega
}v\quad\forall v\in H_{0}^{1}\left(  \Omega\right)  .
\]
It is well known that $u_{\mu}\in H_{0}^{1}\left(  \Omega\right)  \cap
H^{2}\left(  \Omega\right)  $ and $-u_{\mu}^{\prime\prime}=1/a\left(
\mu\right)  $. The linearization of the mapping $\ell$ around $u=u_{\mu}$ can
be written as%
\[
\ell\left(  u_{\mu}+v\right)  =\ell\left(  u_{\mu}\right)  +\left(  D_{\mu
},v\right)  _{\Omega}+\ell\left(  v\right)  \quad\text{for the constant
function }D_{\mu}=24/a\left(  \mu\right)  .
\]
In this way, the solution $v_{\mu}$ of the dual problem in (\ref{amueD}) is
given by%
\[
v_{\mu}=\mathcal{L}_{\mu}^{-1}\left(  D_{\mu}\right)
\]
and belongs to $H^{2}\left(  \Omega\right)  $. Well-known approximation
properties of $\mathbb{P}_{1}$ finite elements then imply%
\[
\eta_{\mu,h}\left(  f\right)  \leq Ch\quad\text{and\quad}\eta_{\mu,h}\left(
D_{\mu}\right)  \leq Ch
\]
and this implies quadratic convergence for the fixed point via Corollary
\ref{Corsuper}.

In Figure \ref{Fig_h_conv}, the errors $\left\Vert \left(  u_{\mu}%
-u_{h,\mu_{h}}\right)  ^{\prime}\right\Vert _{\max}$ and $\left\vert \mu
_{0}-\mu_{h}\right\vert $ are depicted as a function of the number of mesh
points, where $\left\Vert \cdot\right\Vert _{\max}$ denotes a discrete version
of the maximum norm. The expected linear convergence is clearly visible for
$\left\Vert \left(  u_{\mu}-u_{h,\mu_{h}}\right)  ^{\prime}\right\Vert _{\max
}$ as well as the quadratic superconvergence for the fixed point error which
show the sharpness of our theoretical estimates.%
\begin{figure}[tbp] \centering
\begin{tabular}
[c]{lll}%
{\includegraphics[
height=2.0609in,
width=2.5244in
]%
{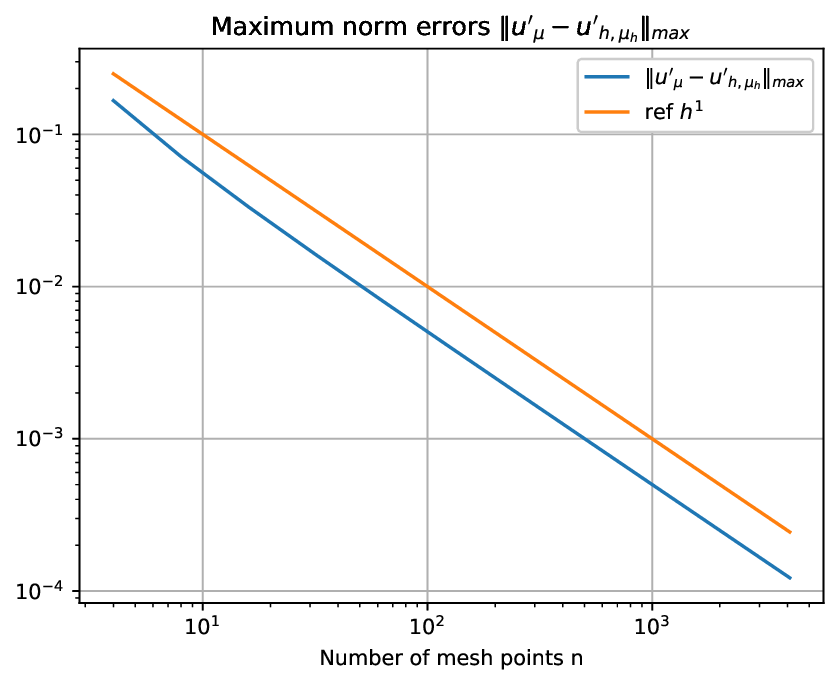}%
}
&  &
{\includegraphics[
height=1.9804in,
width=2.4025in
]%
{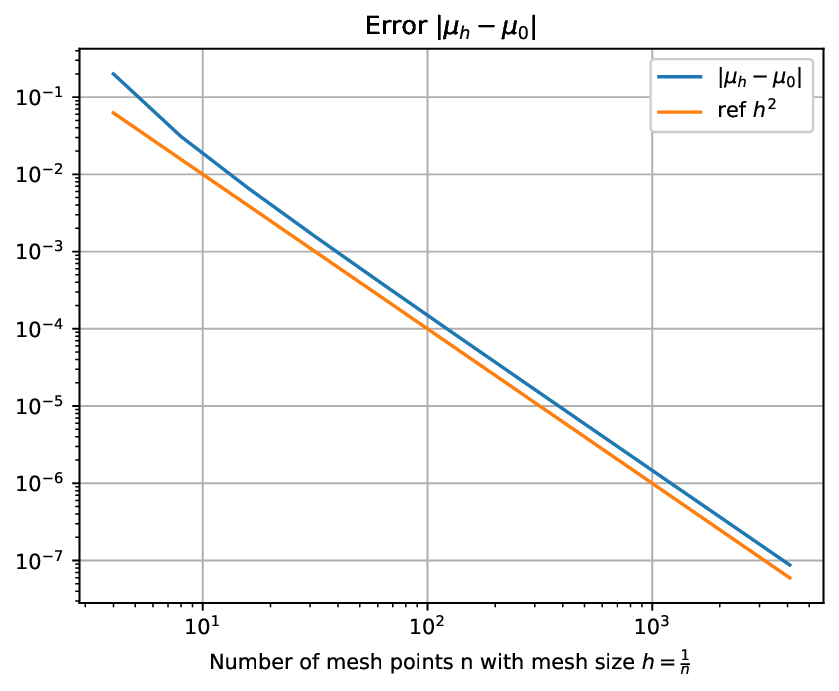}%
}
\end{tabular}
\caption{Convergence with respect to the number of mesh points. The fixed point
iteration is stopped when the machine tolerance has been reached. Left:
Convergence of the solution $u_{h,\mu_{h}}$ towards the exact solution
$u_{\mu}$. Right: Convergence of the discrete fixed point $\mu_{h}$ towards
the exact fixed point $\mu_{0}$.
\label{Fig_h_conv}}%
\end{figure}%

\textbf{Acknowledgment. }Thanks are due to Cl\'{e}ment Canc\`{e}s, Inria
Universit\'{e} de Lille, France, for calling our attention to the modified
fixed point iteration described in Appendix \ref{AppMFI}.

\appendix

\section{A modified fixed point iteration\label{AppMFI}}

In this appendix we present a modified fixed point iteration with improved
convergence properties compared to the simple iteration (\ref{mc16}) provided
some Lipschitz estimate of the function $G$ is available.

Recall that an operator $B:H\rightarrow H$ acting in a Hilbert space $H$ with
scalar product $\left(  \cdot,\cdot\right)  $ and norm $\left\vert
\cdot\right\vert $ is said to be \textit{monotone} iff%
\[
\left(  Bu-Bv,u-v\right)  \geq0\qquad\forall u,v\in H.
\]

\begin{lemma}
\label{LemMon}Let $I:H\rightarrow H$ denote the identity and $A:H\rightarrow
H$ be a mapping such that there exists a constant $\kappa\geq0$ with%
\begin{equation}
\left\vert Au-Av\right\vert \leq\kappa\left\vert u-v\right\vert \qquad\forall
u,v\in H. \label{App1}%
\end{equation}
Then $B_{\pm}=\kappa I\pm A$ is monotone.
\end{lemma}

%

\proof
The proof follows from the estimate%
\begin{align*}
\left(  \left(  \kappa u\pm Au\right)  -\left(  \kappa v\pm Av\right)
,u-v\right)   &  =\kappa\left\vert u-v\right\vert ^{2}\pm\left(
Au-Av,u-v\right) \\
&  \geq\kappa\left\vert u-v\right\vert ^{2}-\left\vert Au-Av\right\vert
\left\vert u-v\right\vert \\
&  =\left\vert u-v\right\vert \left(  \kappa\left\vert u-v\right\vert
-\left\vert Au-Av\right\vert \right)  \overset{\text{(\ref{App1})}}{\geq
}0\quad\forall u,v\in H.
\end{align*}%
\endproof

Let $A:\mathbb{R}\rightarrow\mathbb{R}$ satisfying (\ref{App1}) and let
$u_{\ast}\in\mathbb{R}$ denote a fixed point $u_{\ast}=Au_{\ast}$. This
equality is equivalent to $\left(  \kappa+1\right)  u_{\ast}=\left(  \kappa
I+A\right)  u_{\ast}$ and, in turn, to
\begin{equation}
u_{\ast}=\frac{1}{\kappa+1}\left(  \kappa I+A\right)  u_{\ast},
\label{Appfixedp2}%
\end{equation}
i.e., finding a fixed point to $A$ is equivalent to find one to
(\ref{Appfixedp2}).

Equation (\ref{Appfixedp2}) induces the fixed point iteration%
\begin{equation}
\left\{
\begin{array}
[c]{ll}%
u_{0} & \text{given,}\\
u_{n+1}=\dfrac{1}{\kappa+1}\left(  \kappa I+A\right)  u_{n} & n=0,1,\ldots.
\end{array}
\right.  \label{fixedpointit2}%
\end{equation}

\begin{lemma}
\label{LemMFI}Let $A:\mathbb{R}\rightarrow\mathbb{R}$ satisfy (\ref{App1}) for
some $\kappa\geq0$ and $u_{\ast}\in\mathbb{R}$ be a fixed point $u_{\ast
}=Au_{\ast}$. We assume that one of the following conditions hold

\begin{enumerate}
\item $u_{0}\leq u_{1}$ and $u_{0}\leq u_{\ast},$

\item $u_{0}\geq u_{1}$ and $u_{0}\geq u_{\ast}.$
\end{enumerate}

Then, the fixed point iteration (\ref{fixedpointit2}) converges.
\end{lemma}

%

\proof
We restrict to the case $u_{0}\leq u_{\ast}$ and $u_{0}\leq u_{1}$, while the
convergence proof for the reversed case 2. is verbatim.

First we prove by induction: if $u_{0}\leq u_{\ast}$ then $u_{n}\leq u_{\ast}$
for all $n$. Indeed,%
\begin{align*}
u_{\ast}-u_{n+1}  &  =\frac{1}{\kappa+1}\left(  \kappa\left(  u_{\ast}%
-u_{n}\right)  +A\left(  u_{\ast}\right)  -A\left(  u_{n}\right)  \right) \\
&  \geq\frac{1}{\kappa+1}\left(  \kappa\left(  u_{\ast}-u_{n}\right)
-\kappa\left\vert u_{\ast}-u_{n}\right\vert \right)  .
\end{align*}
By induction we know that $\left\vert u_{\ast}-u_{n}\right\vert =u_{\ast
}-u_{n}$ and $u_{\ast}-u_{n+1}\geq0$ follows.

Secondly, we conclude from $u_{0}\leq u_{1}$ that $u_{n}\leq u_{n+1}$ for all
$n\geq0$: assume by induction that $u_{n-1}\leq u_{n}$. The next iterate is
given for $B=\kappa I+A$ by%
\[
u_{n+1}=\frac{1}{\kappa+1}Bu_{n}.
\]
Since $B$ is monotone we have%
\[
\left(  Bu_{n}-Bu_{n-1}\right)  \left(  u_{n}-u_{n-1}\right)  \geq0.
\]
The induction assumption leads to $Bu_{n}\geq Bu_{n-1}$ and, in turn, to%
\[
u_{n+1}=\frac{1}{\kappa+1}Bu_{n}\geq\frac{1}{\kappa+1}Bu_{n-1}=u_{n}.
\]
We have proved that $\left(  u_{n}\right)  _{n}$ is monotonically increasing
and bounded from above by $u_{\ast}$; this implies convergence.%
\endproof

\begin{remark}
For our application it holds $A\left(  \mu\right)  \leftarrow G\left(
\mu\right)  =\ell\left(  u_{\mu}\right)  $ or $A\left(  \mu\right)  \leftarrow
G_{S}\left(  \mu\right)  =\ell\left(  u_{S,\mu}\right)  $. Hence, the
iteration (\ref{Appfixedp2}) is converging if the Lipschitz constant $\kappa$
in (\ref{App1}) (or an upper bound thereof) is known and one of the two cases
stated in Lemma \ref{LemMFI} hold.
\end{remark}

\eject

\end{document}